\documentclass[11pt]{amsart}
\usepackage{amssymb}

 \pdfoutput=1
\usepackage{hyperref} 
\usepackage[dvips]{graphicx}
\usepackage[T1]{fontenc}
\usepackage{mathptmx}

\usepackage{color}

\DeclareMathOperator{\trace}{trace}
\DeclareMathOperator{\Ad}{Ad}

\newcommand{\LL}{\mathbb{L}}
\newcommand{\bbar}{\begin{pmatrix}}
\newcommand{\ebar}{\end{pmatrix}}

\newcommand{\enormal}{{\bf n}_E}
\newcommand{\G}{\mathcal{G}}

\newcommand{\gggg}{\mathcal{G}}
\newcommand{\bdm}{\begin{displaymath}}
\newcommand{\edm}{\end{displaymath}}
\newcommand{\beq}{\begin{equation}}
\newcommand{\beqa}{\begin{eqnarray}}
\newcommand{\beqas}{\begin{eqnarray*}}
\newcommand{\eeq}{\end{equation}}
\newcommand{\eeqa}{\end{eqnarray}}
\newcommand{\eeqas}{\end{eqnarray*}}
\newcommand{\dd}{\textup{d}}

\newcommand{\E}{{\mathbb E}}
\newcommand{\R}{{\mathbb R}}
\newcommand{\C}{{\mathbb C}}
\newcommand{\D}{{\mathbb D}}

\newcommand{\PP}{\mathcal{P}}
\newcommand{\real}{{\mathbb R}}
\newcommand{\SSS}{{\mathbb S}}

\newcommand{\sym}{\mathcal{S}}
\newcommand{\B}{\mathcal{B}}
\newcommand{\Z}{{\mathbb Z}}

\newcommand{\ip}[2]{\langle#1,#2\rangle}
\newcommand{\ipb}[2]{\biggl\langle#1,#2\biggr\rangle}

\newcommand{\sign}{\mathrm{sign}}

\newcommand{\SLR}{\mathrm{SL}(2,\real)}
\newcommand{\slR}{\mathfrak{sl}(2,\real)}
\newcommand{\SLC}{\mathrm{SL}(2,\C)}
\newcommand{\slC}{\mathfrak{sl}(2,\C)}

   \newtheorem{theorem}{Theorem}[section]
   \newtheorem{proposition}[theorem]{Proposition}
   
   \newtheorem{lemma}[theorem]{Lemma}
   \newtheorem{definition}[theorem]{Definition}
  \newtheorem{problem}[theorem]{Problem}

 \theoremstyle{remark}
   
   \newtheorem{remark}[theorem]{Remark}
   
\numberwithin{equation}{section}

\begin{document}

\title[Timelike CMC surfaces with Singularities]{Timelike Constant Mean Curvature Surfaces with Singularities}

\begin{abstract}
We use integrable systems techniques to study the singularities of timelike non-minimal constant mean curvature (CMC) surfaces in the
Lorentz-Minkowski 3-space.  The singularities arise at the boundary
of the Birkhoff big cell of the loop group involved. We examine the 
behaviour of the surfaces at the big cell boundary, generalize the
definition of CMC surfaces to include those with finite, generic 
singularities,
and show how to construct surfaces with prescribed singularities by
solving a singular geometric Cauchy problem.  The solution shows that
the generic singularities of the generalized surfaces are 
cuspidal edges, swallowtails and cuspidal cross caps.
\end{abstract}

\author{David Brander}
\address{Department of Mathematics\\ Matematiktorvet, Building 303 S\\
Technical University of Denmark\\
DK-2800 Kgs. Lyngby\\ Denmark}
\email{D.Brander@mat.dtu.dk}

\author{Martin Svensson}
\address{Department of Mathematics \& Computer Science\\ and CP3-Origins, Centre of Excellence for Particle Physics Phenomenology\\
  University of Southern Denmark\\ Campusvej 55\\ DK-5230 Odense M\\
   Denmark}
\email{svensson@imada.sdu.dk}



\thanks{Research partially sponsored by FNU grant \emph{Symmetry Techniques in Differential Geometry}, and by CP3-Origins DNRF 
Centre of Excellence Particle Physics Phenomenology. Report no. 
CP3-Origins-2011-32 \& Dias-2011-24.}

\maketitle

\section{Introduction}
\noindent
The study of singularities of timelike constant mean curvature (CMC) surfaces in Lorentz-Minkowski
 $3$-space $\LL^3$, initiated in this article, has two contexts in current research:  One context is the use of loop group techniques in geometry, 
 whereby special submanifolds are constructed from, or represented by, simple data via loop group decompositions.
  When the underlying Lie group is non-compact the decomposition used in the
 construction breaks down on certain lower dimensional subvarieties. It is of interest to understand 
 what effect this has on the special submanifold.   \\
 
\noindent
 The second context is the study of surfaces with singularities. This has gained some attention in recent years: see, for example 
  \cite{fl2007, fls2005, fsuy, kimkohshinyang2011,  kimyang2007,
     umsa2009,  suy, umeda} and related works.
 Singularities arise naturally and frequently in geometry: one motivation for
 their study is that many surface classes have either no, or essentially no,
 complete regular examples, the most famous case being pseudospherical surfaces.
 One generalizes the definition of a surface to that of a \emph{frontal}, a map which is immersed on an open dense subset of the domain and has a well defined unit normal everywhere.  A basic question is to find the generic singularities for a given surface class.
For example,  the generic singularities of constant Gauss curvature surfaces in Euclidean $3$-space are cuspidal edges and swallowtails \cite{ishimach}, whilst spacelike
mean curvature zero surfaces in $\LL^3$ have cuspidal cross caps  in
 addition to the two singularities just mentioned \cite{umyam2006}.  
 The point is that different geometries have different 
generic singularities. The first-named of the present authors studied singularities of spacelike 
non-zero CMC surfaces in
 $\LL^3$ in \cite{sbjorling}, of which more below, but the singularities of \emph{timelike} CMC surfaces appear to be uninvestigated.\\

\noindent
In the loop group context, solutions are generally obtained via either the \emph{Iwasawa} decomposition
$\Lambda G^\C = \Omega G \, \cdot \, \Lambda ^+ G^\C$, a situation which includes harmonic maps
into symmetric spaces,  or via the \emph{Birkhoff} decomposition $\Lambda G = \Lambda^- G \cdot \Lambda ^+ G$,
a situation which includes \emph{Lorentzian} harmonic maps into Riemannian symmetric spaces.  Both of these
types of harmonic maps correspond to various well-known surfaces classes, such as constant Gauss or mean curvature surfaces in space forms -- for surveys of some of these, see \cite{bobenko1994, dorfsurvey}.  When the real form $G$ is non-compact, the left hand side of the decomposition is replaced by an open dense subset, the \emph{big cell}, of the loop group, rather than the whole.  Since, at the global level,
 there is no general way to avoid the big cell
boundary,  there remains the question of what happens to the surface at 
this boundary.  \\

\noindent
The Riemannian-harmonic (Iwasawa) case was investigated in \cite{sbjorling, brs}, through the 
study of spacelike CMC surfaces in $\LL^3$.  The big cell boundary is a disjoint union
 $\mathcal{P}_{\pm 1} \cup \mathcal{P}_{\pm 2} \cup ....$ of smaller cells, with increasing codimension.
The lowest codimension small cells, $\mathcal{P}_{\pm 1}$, where generic singularities would occur, were analyzed, and it was found that finite singularities occur on one of these, whilst the surface blows up
at the other.  In \cite{sbjorling} a singular Bj\"orling construction was devised to construct prescribed singularities, and the generic singularities for the generalized surface class
defined there were found to be 
cuspidal edges, swallowtails and cuspidal cross caps. \\

\noindent
In the present work, we turn to the Lorentzian-harmonic (Birkhoff) situation, and study the example of 
\emph{timelike} CMC surfaces.  The loop group construction differs from the spacelike case in that
the basic data are now two functions of one variable, rather than the one holomorphic function of
the Riemannian harmonic case.  The   Birkhoff decomposition construction  compared to
the Iwasawa construction, as well as the hyperbolic as opposed to elliptic nature of the problem,
pose new challenges.  However, we  obtain analogous results to those of the spacelike case in \cite{sbjorling} and \cite{brs}.\\

\subsection{Results of this article}
The generalized d'Alembert representation used here, which was 
given by Dorfmeister, Inoguchi and Toda \cite{dit2000}, allows one to construct all timelike CMC
surfaces from  pairs of functions of one variable $\hat X(x)$ and $\hat Y(y)$ which take values in a certain real
form of the loop group $\Lambda G^\C$, where $G=\SLR$. The construction  depends crucially on 
a pointwise Birkhoff decomposition of the map $\hat \Phi(x,y) := \hat X^{-1}(x) \hat Y(y)$.  
The data are thus at the big cell boundary at $z_0 = (x_0,y_0)$ if $\hat \Phi(z_0)$ is not in
the Birkhoff big cell.  The complement  of the big cell is a disjoint union
$\bigcup_{j=1}^\infty \mathcal{P}_L^{\pm j}$ of subvarieties.  The codimension of the small cells increases
with $|j|$, and therefore generic singularities should occur only on $\mathcal{P}_L^{\pm1}$.
We prove in Theorem \ref{firstsmallcellthm} that if $\hat \Phi(z_0) \in \mathcal{P}_L^1$ then
the surface has a finite singularity, and at $\mathcal{P}_L^{-1}$ 
 the surface blows up. 
 To investigate the type of the finite singularity, we define generalized timelike
CMC surfaces to be surfaces that can locally be represented by d'Alembert data which maps
into the union of the the big cell and $\mathcal{P}_L^1$.\\

\noindent
 We restrict the discussion to singularities that are  \emph{semi-regular}, that is, where the differential of the surface $f$ has rank $1$, a condition that can be prescribed in the data $\hat X$ and $\hat Y$.  These surfaces are frontals, and there is a well defined (up to local choice of orientation) Euclidean unit normal
$\enormal$, which can be locally expressed by $f_x \times f_y = \chi \enormal$. The function 
$\chi$ obviously vanishes at points where $f$ is not immersed, and we generally study 
\emph{non-degenerate} singularities, that is, points where $\dd \chi \neq 0$.\\

\noindent
On generalized timelike CMC surfaces, one finds that singularities come in two classes, which we call \emph{class I} and \emph{class II}, respectively characterized geometrically by 
the property that the direction $\enormal$ is \emph{not} or \emph{is} lightlike in $\LL^3$. 
Class I singularities never occur at the big cell boundary, but rather due to one of the maps
 $\hat X$ or $\hat Y$ not
satisfying the regularity condition for a smooth surface.  We discuss these singularities in Section \ref{classIsect}
and prove that the generic singularities
are cuspidal edges. 
 Such singularities can easily be prescribed by 
choosing $\hat X$ and $\hat Y$ accordingly, but there is no unique solution for the Cauchy
problem for such  a singular curve, because it is always a \emph{characteristic} curve for the underlying PDE. \\

\noindent
Class II singularities, on the other hand, always occur at the big cell boundary, and are the real object of interest in this article.  Note that, although in this case the tangent to the singular curve is lightlike, this does \emph{not} mean that the curve is characteristic in the coordinate domain, in contrast to the situation on an immersed surface.  The curve can be either non-characteristic  or characteristic, and  
generic non-degenerate singularities, studied in Section \ref{noncharsect},
 are non-characteristic.  In Section \ref{singpotsection}, we prove that all generalized 
 timelike CMC surfaces with non-characteristic class II singular curves can be produced 
 by certain "singular potentials". \\

\noindent
 In Section \ref{singgcpsect}, 
 Theorem \ref{singgcptheorem}, we find the singular potentials which solve the \emph{non-characteristic singular geometric Cauchy problem}, (Problem \ref{ncsgcpprob}),
 which is to find the generalized timelike CMC surface with prescribed non-characteristic 
 singular curve, and an additional (geometrically relevant) vector field prescribed along the curve.  The \emph{non}-singular version of this problem was solved in \cite{dbms1}, using
the generalized d'Alembert setup.  
It is not possible to 
apply the non-singular solution to the singular case because the solution depends on the construction
of an $\SLR$ frame for the surface, along the curve, directly from the geometric
Cauchy data.  However, the $\SLR$ frame blows up and is not defined at the big cell boundary,
necessitating a work-around. \\

\noindent
The solution of the singular geometric Cauchy problem is critical to the study of generic singularities in Section \ref{gensingsect}. 
The geometric Cauchy data consists of three functions
$s(v)$, $t(v)$ and $\theta(v)$ along a curve, which are more or less arbitrary.  The singularity
at the point $v=0$ is non-degenerate if and only if $\theta^\prime (0)\neq 0$ and $s (0)\neq \pm t(0)$.  Given this
assumption,  the main result of this section, Theorem \ref{singtheorem}, states that we have the following correspondences:
\beqas
 \textup{cuspidal edge}  & \leftrightarrow &  s(0)\neq 0\neq t(0), \\
  \textup{swallowtail}  & \leftrightarrow &  s(0)=0, 
      \textup{ and }  s^\prime (0) \neq0,\\
 \textup{cuspidal cross cap} & \leftrightarrow &  t(0) =0, 
      \textup{ and }  t^\prime(0) \neq 0.
 \eeqas
   This shows that the generic non-degenerate singularities are
just these three, since the only other possibility is a higher order zero.\\
 
\begin{figure}[here]  
\begin{center}
\includegraphics[height=40mm]{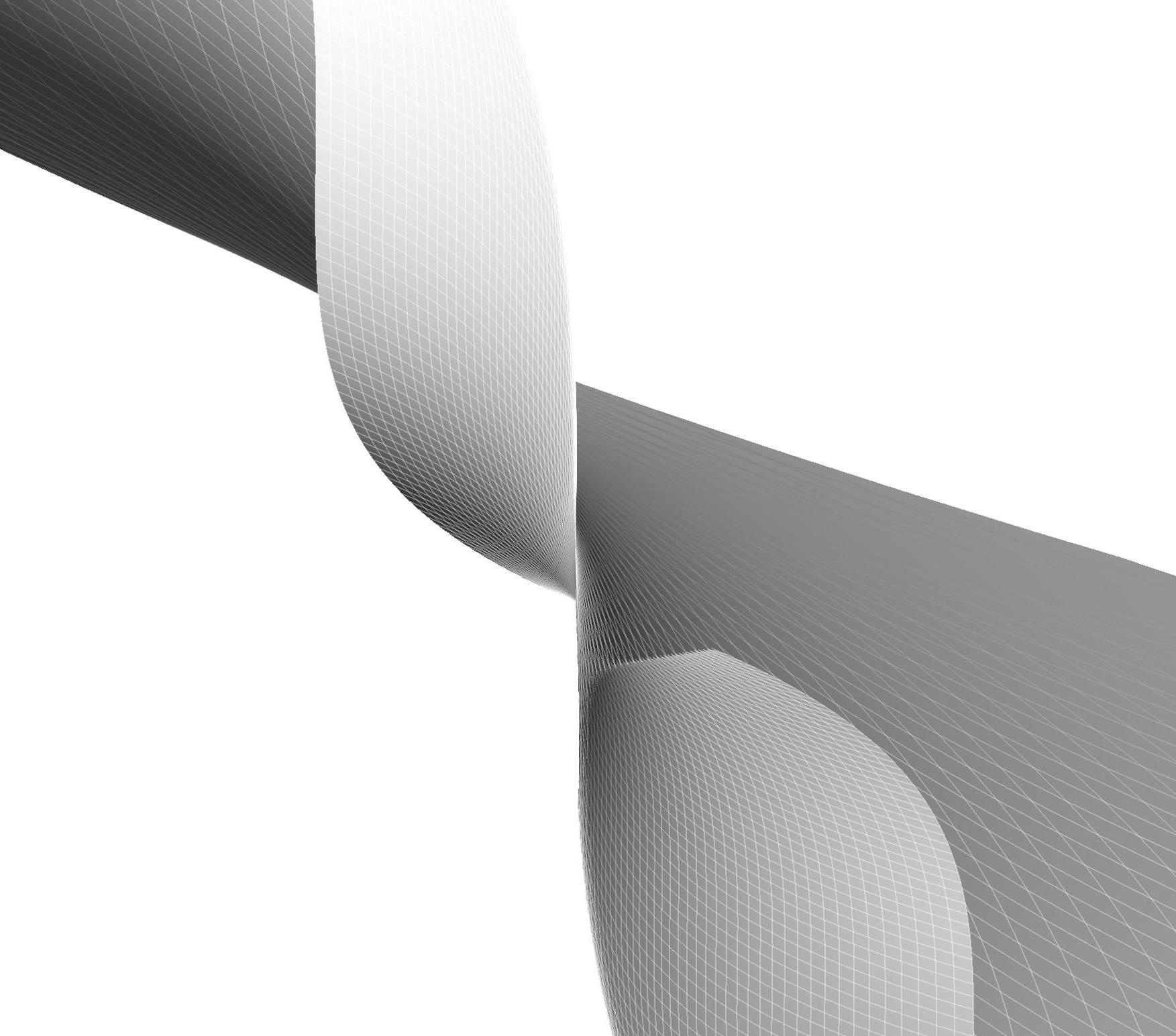} \hspace{2cm}
\includegraphics[height=40mm]{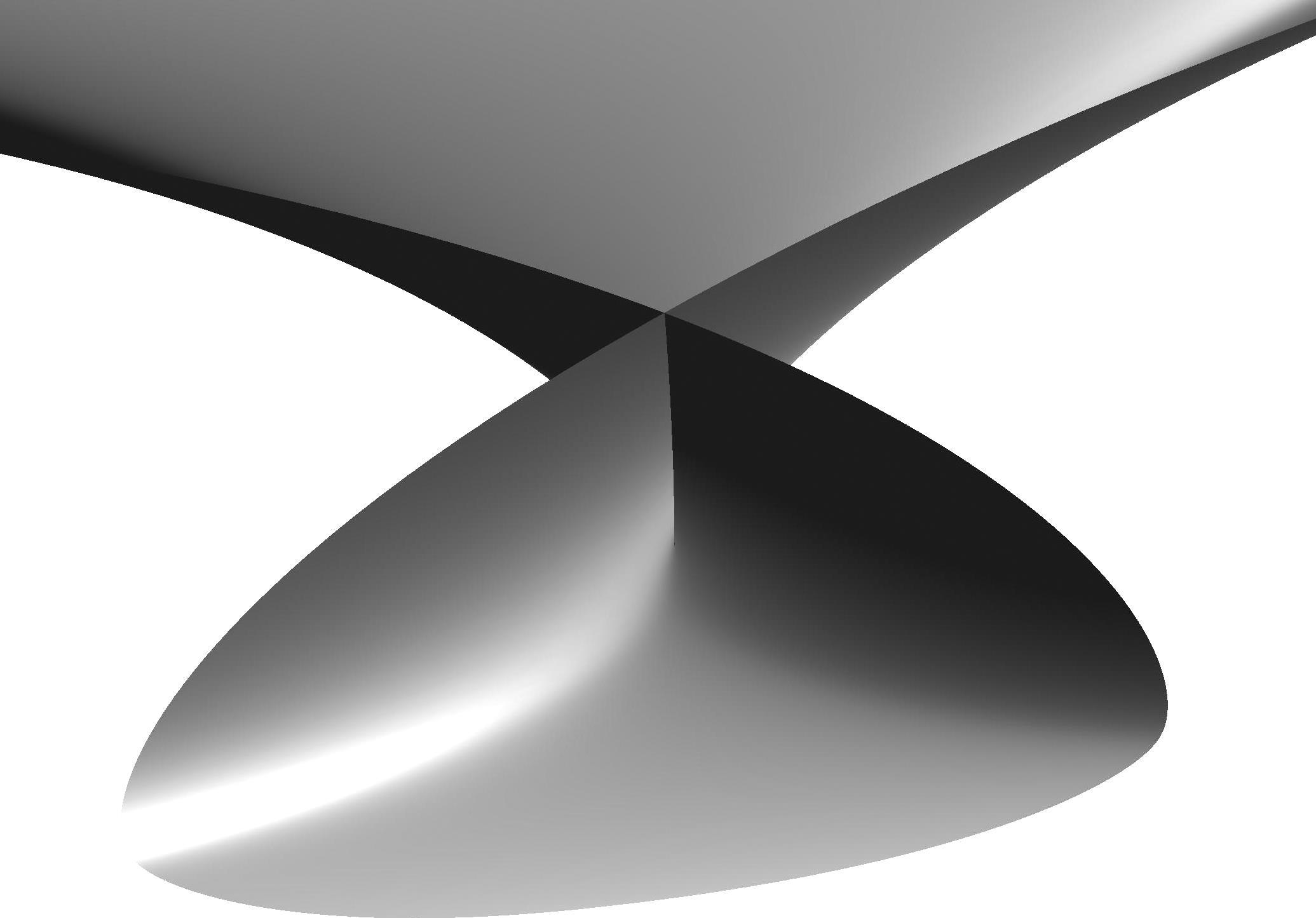} 
\end{center}
\caption{Numerical plots of solutions to the geometric Cauchy problem.
Left: $s(v)= 2+0.2v^2$, $t(v) = v$ (cuspidal cross cap). Right: $s(v)=v$, $t(v) =1$ (swallowtail).}
\label{figure1}
\end{figure}

\noindent
In the last two sections we consider non-generic singularities.  In Section
 \ref{chartypesect} we solve the geometric Cauchy problem for \emph{characteristic} data,
 where there are infinitely many solutions.  The singular curve is always a straight line
 in this case.   In Section \ref{numsect} we compute numerically some examples of degenerate singularities.\\

\noindent
In conclusion, we remark that the
 results of this article, combined with the results on Riemannian harmonic maps in 
\cite{sbjorling, brs}, ought to give a good indication of the typical situation at the 
big cell boundary for surfaces associated to harmonic or Lorentzian harmonic maps.\\

 \noindent
\textbf{Notation:} If $\hat X$ is a map into a loop group or loop algebra, 
we will sometimes use $X^\lambda$ for the corresponding group or algebra valued map 
$\hat X \big |_ \lambda$, obtained
by evaluating at a particular value $\lambda$ of the loop parameter. We also use $X := X^1$.
We use $\ip{\cdot}{\cdot}_E$ and $\ip{\cdot}{\cdot}_L$ for Euclidean and
Lorentzian inner products respectively.  
 We use $O(\lambda^k)$ for an expression $g(\lambda)$ such that $\lim_{\lambda \to 0} g(\lambda)/\lambda^k$ is finite, and
  $O_\infty(\lambda^k)$ for the anlogue when $\lambda \to \infty$.\\  


\section{Background material}
\noindent
We give a brief summary of the 
 method given by Dorfmeister, Inoguchi and Toda \cite{dit2000} for constructing all timelike CMC surfaces from pairs of functions of one variable. The conventions we will use are mostly the same as those we used in \cite{dbms1}, and the reader is therefore referred to that article for more details of the following sketch.\\

\subsection{Loop groups}
Let $ \G=\Lambda\SLC_{\sigma\rho}$ be  the
 group of loops in $\SLC$, with loop parameter $\lambda$, that  are fixed by the commuting involutions 
$$
(\rho \gamma) (\lambda) = \overline{\gamma(\bar \lambda)}, \quad
(\sigma \gamma) (\lambda) = \Ad_{P} \gamma(-\lambda), \quad 
P =  \textup{diag}(1,-1).
$$
The group $\G$ is a real form  of 
$\G^\C=\Lambda\SLC_\sigma$, the group of loops
fixed by $\sigma$.\\

\noindent
Let $\Lambda^\pm\SLC_\sigma$ denote the subgroup of $\G^\C$ consisting of loops that extend holomorphically to $\D^\pm$, where $\D^+$ is the unit disc and $\D^- = \SSS^2 \setminus \{ \D^+ \cup \SSS^1 \}$, the exterior disc in the Riemann sphere. Define 
$$
\G^{\pm}=\G\cap\Lambda^\pm\SLC_\sigma,\quad \G^+_*=\{\gamma\in\G^+\ |\ \gamma(0)=I\},\quad \G^-_*=\{\gamma\in\G^-\ |\ \gamma(\infty)=I\}.
$$
We define the complex versions $\G^{\C\pm}$ analogously by substituting $\G^\C$ for $\G$ in the above definitions.\\

\noindent
The essential tool from loop groups needed is the Birkhoff decomposition,
 due to Pressley and Segal \cite{PreS}. See \cite{branderdorf} for a more general statement which includes the following case:
\begin{theorem}[The Birkhoff decomposition]\label{birkhoff}$\phantom{}$
The sets $\B_L=\G^-\cdot\G^+$ and $\B_R=\G^+\cdot\G^-$ are both open and dense in $\G$.  The multiplication maps 
$$
\G^-_*\times\G^+\to\B_L \ \text{  and  }\ \G^+_*\times\G^-\to\B_R
$$
are both real analytic diffeomorphisms.
\end{theorem}
Note that the analogue also holds, substituting $\G^\C$, $\G^{\C \pm}$ and $\G^{\C \pm}_*$ for $\G$, $\G^\pm$ and $\G^\pm_*$, respectively, writing $\B_L^\C=\G^{\C-}\cdot\G^{\C+}$ and $B_R^\C=\G^{\C+}\cdot\G^{\C-}$. \\

\noindent
The basis of the loop group approach is that timelike CMC surfaces correspond
to a particular type of map into $\G$:
\begin{definition}  \label{admissibleframedef}
Let $M$ be a simply connected open subset of $\real^2$, and let $(x,y)$ denote the standard 
coordinates. An \emph{admissible frame} on $M$ is a smooth map $\hat F: M \to \G$  such that the Maurer-Cartan form of $\hat F$ is a Laurent polynomial in $\lambda$ of the form:
\bdm
\hat F^{-1} \dd \hat F =\lambda\, A_{1}\, \dd x + \alpha_0 + \lambda^{-1}A_{-1}\, \dd y,
\edm
where the $\slR$-valued 1-form $\alpha_0$ is constant in $\lambda$.  The admissible frame
 $\hat F$ is said to be \emph{regular} if the components $[A_1]_{21}$ and $[A_{-1}]_{12}$ are non-vanishing. \\
\end{definition}

\subsection{Timelike CMC surfaces as admissible frames}
We identify the Lie algebra $\slR$ with Lorentz-Minkowski space $\LL^3$, with basis:
\begin{eqnarray*}  
e_0 = \bbar 0 & -1 \\1 & 0\ebar, & e_1 = \bbar 0 & 1 \\1 & 0\ebar, & 
    e_2 = \bbar -1 & 0 \\ 0 & 1\ebar,
\end{eqnarray*}
which are orthonormal with respect to the inner product
 $\ip{X}{Y}_L=\frac{1}{2}\trace(XY)$,  and with $\ip {e_0}{e_0}_L=-1$.\\

\noindent
Let $M$ be a simply connected domain in $\real^2$, and 
$f:M\to\LL^3$ a  timelike immersion. The induced metric determines a Lorentz conformal structure on $M$. For any lightlike (also called null)
 coordinate system $(x,y)$ on $M$, we define a function $\omega: M \to \real$ by the condition that the induced metric is given by
\bdm
\dd s^2 = \varepsilon  e^\omega \, \dd x \, \dd y, \hspace{1cm} \varepsilon= \pm 1.
\edm
Let $N$ be a unit normal field for the immersion $f$,  and define a \emph{coordinate frame} for $f$ to be a map $F:M\to\SLR$ which satisfies
\bdm 
f_x = \dfrac{\varepsilon_1}{2} e^{\omega/2} \Ad_F (e_0 + e_1), \quad
f_y = \dfrac{\varepsilon_2}{2} e^{\omega/2} \Ad_F (-e_0 + e_1), \quad
N = \Ad_F(e_2),
\edm
where $\varepsilon_1,\varepsilon_2\in\{-1,1\}$, so that $\dd s^2$
is as above with $\varepsilon=\varepsilon_1\varepsilon_2$. Conversely, since $M$ is simply connected, we can always construct a coordinate frame for a timelike conformal immersion $f$.\\

\noindent
The Maurer-Cartan form $\alpha$ for the frame $F$ is defined by
\bdm
\alpha = F^{-1} \dd F = U \dd x + V \dd y = A_1 \dd x + \alpha_0 + A_{-1} \dd y,
\edm
 where $A_{\pm 1}$ are off-diagonal and $\alpha_0$ is a diagonal matrix valued 1-form. Let $\textup{Lie}(X)$ denote the Lie algebra of any group $X$.   We extend $\alpha$ to a  
   $\textup{Lie} (\G)$-valued $1$-form $\hat \alpha$ by inserting the paramater $\lambda$
   as follows:
\bdm
\hat \alpha = A_1 \lambda \dd x + \alpha_0 + A_{-1} \lambda^{-1} \dd y,
\edm
where $\lambda$ is the complex loop parameter.
 The surface $f$ is  of constant mean curvature if and only
 if $\hat \alpha$ satisfies the Maurer-Cartan equation
 $\dd \hat \alpha + \hat \alpha \wedge \hat \alpha = 0$, and one can
then integrate the  equation $\hat F^{-1} \dd \hat F = \hat \alpha$, with $\hat F(0)=I$,
 to obtain the \emph{extended coordinate frame}
  $\hat F: M \to \G$, which is a regular admissible frame.\\

\noindent
It is important to note that the $1$-forms $A_1 \dd x$, $A_{-1} \dd y$ and $\alpha_0$ are well-defined, independently of the choice of (oriented) lightlike coordinates, because any other lightlike coordinate system with the same orientation is 
given by $(\tilde x(x,y), \tilde y(x,y)) = (\tilde x (x), \tilde y (y))$.  This means that 
the extension of $F$ to $\hat F$ does not depend on coordinates. \\

\noindent
One can reconstruct the surface $f$ as follows: define the map 
 $\sym: \Lambda \SLC \to\Lambda\slC$, 
$$
\sym(\hat G)=2\lambda\partial_\lambda\hat G\hat G^{-1}-\Ad_{\hat G}(e_2). 
$$
For any $\lambda_0\neq0$, define $\sym_{\lambda_0}(\hat G)=\sym(\hat G)\big\vert_{\lambda=\lambda_0}$. 
Assume coordinates are chosen such that  $f(p) = 0$ for some point
$p \in M$.  Then $f$ is recovered by the \emph{Sym formula}
$$
f (z) = \frac{1}{2H}\left \{ \sym_1 (\hat F(z)) - \sym_1(\hat F(p))\right \}.\\
$$

\noindent
Conversely, \emph{every} regular admissible frame gives a timelike CMC surface:
first note that a regular admissible frame can be written
 $\hat F^{-1} \dd \hat F = \hat U \dd x + \hat V \dd y$, with
  $$
 \hat U = \bbar a_1 & b_1 \lambda \\ c_1 \lambda & -a_1 \ebar\quad \text{   and   }\quad
 \hat V = \bbar a_2 & b_2 \lambda^{-1} \\ c_2 \lambda^{-1} & -a_2 \ebar.
 $$
 where $c_1$ and $b_2$ are non-zero.

\begin{proposition}  \label{characterizationprop}
Let $\hat F: M \to \G$ be a regular admissible frame and $H\neq0$. 
Set $\varepsilon_1=\sign(c_1)$, $\varepsilon_2=-\sign(b_2)$ and $\varepsilon=\varepsilon_1\varepsilon_2$. Define a Lorentz metric on $M$ by
$$
\dd s^2 = \varepsilon e^\omega \dd x \, \dd y, \quad
\varepsilon e^\omega = -\dfrac{4 c_1 b_2}{H^2}.
$$
Set
$$
f^\lambda =\frac{1}{2H}\sym_{\lambda}(\hat F):M\to\LL^3\qquad(\lambda\in\real\setminus\{0\}).
$$
Then, with respect to the choice of unit normal $N^\lambda = \Ad_{\hat F} e_2$, and the given metric, the surface $f^\lambda$ is a timelike CMC $H$-surface. Set
$$
\rho=\left|\frac{b_2}{c_1}\right|^{\frac{1}{4}},\qquad T=\bbar \rho & 0 \\ 0 & \rho^{-1} \ebar,  
$$
and set $\hat F_C = \hat F T : M \to \G$. Then $\hat F_C$ is the 
extended coordinate frame for the surface $f=f^1$. For general values of $\lambda \in \real \setminus \{0\}$ we have:
\beq\label{signedcoordframe}
\begin{split}
f^{\lambda} _x =&  \frac{ \lambda c_1 \rho^2}{H} \Ad_{F^\lambda _C}(e_1 + e_0)
=
  \frac{\lambda \varepsilon_1 e^{\omega/2}}{2}  \Ad_{F^\lambda _C}(e_1 + e_0) , \\
f^{\lambda} _y =&  \frac{b_2 \rho^{-2}}{\lambda H}
\Ad_{F^\lambda _C} (e_1-e_0) = 
  \frac{\varepsilon_2 e^{\omega/2}}{2 \lambda} \Ad_{F^\lambda _C} (e_1 -e_0), \\
N^{\lambda} =& \Ad_{F^\lambda _C} e_2 = \Ad_{F^\lambda } e_2,
\end{split}
\eeq 
where $N^\lambda$ is the unit normal to $f^\lambda$.\\
\end{proposition}

\subsection{The d'Alembert type construction}\label{dpwsection}
 We now explain how to construct all admissible frames, and thereby all timelike CMC surfaces, from simple data.

\begin{definition} \label{potentialdefn}
Let $I_x \subset \real$ and $I_y \subset \real$ be open sets, with coordinates $x$ and $y$, respectively.   A \emph{potential pair} $(\psi^X, \psi^Y)$ is a pair of smooth 
$\textup{Lie}(\G)$-valued 1-forms on $I_x$ and $I_y$ respectively with Fourier expansions in $\lambda$ as follows:
\begin{eqnarray*}
\psi^X = \sum_{j=-\infty}^1 \psi^X_i \lambda^i \dd x,\qquad \psi^Y = \sum_{j=-1}^\infty \psi^Y_i \lambda^i \dd y.
\end{eqnarray*}
The potential pair is called \emph{regular} at a point $(x,y)$ if $[\psi^X_1]_{21}(x)\neq 0$ and $[\psi^Y_{-1}]_{12}(y)\neq0$, and \emph{semi-regular} if at most one of these functions
vanishes at $(x,y)$, and the zero is of first order. The pair is called regular or semi-regular
if the corresponding property holds at all points in $I_x \times I_y$.\\
\end{definition}

\noindent
The following theorem is a straightforward consequence of Theorem \ref{birkhoff}. 
Note that the potential pair in Item \ref{dpwthmitem1} of the theorem is well defined, independent of the choice of lightlike coordinates:
\begin{theorem} \label{dpwthm}
\begin{enumerate}
\item \label{dpwthmitem1}
Let $M$ be a simply connected subset of $\real^2$ and $\hat F: M \to \B \subset \G$ an admissible frame.  The pointwise (on $M$) Birkhoff decomposition 
\bdm 
\hat F = \hat Y_- \hat H_+ = \hat X_+ \hat H_-,
\edm
 where 
  $\hat Y_-(y) \in \G^-_*$, $\hat X_+(x) \in \G^+_*$,  and  
  $\hat H_\pm(x,y) \in \G^\pm$,
results in a potential pair $(\hat X_+^{-1} \dd \hat X_+ \, , \, \hat Y_-^{-1} \dd \hat Y_-)$,
of the form
$$
\hat X_+^{-1} \dd \hat X_+ = \psi^X_1 \lambda \, \dd x,\qquad
\hat Y_-^{-1} \dd \hat Y_- = \psi^Y_{-1} \lambda^{-1} \, \dd y.\\
$$
\item
Conversely, given any potential pair, $(\psi^X, \psi^Y)$, define $\hat X: I_x \to \G$ and $\hat Y: I_y \to \G$ by  integrating the differential equations
\begin{eqnarray*}
\hat X^{-1} \dd \hat X = \psi^X, & \hat X(x_0) = I,\\
\hat Y^{-1} \dd \hat Y = \psi^Y, & \hat Y(y_0) = I.
\end{eqnarray*}
Define $\hat \Phi=\hat X^{-1} \hat Y : I_x \times I_y \to \G$, and set 
$M = \hat \Phi^{-1}(\B_L)$. Pointwise on $M$, perform the Birkhoff decomposition $\hat \Phi = \hat H_- \hat H_+$, where $\hat H_-: M \to \gggg^-_*$ and $\hat H_+ : M \to \G^+$.
Then $\hat F = \hat Y \hat H_+^{-1}$ is an admissible frame. \\
\item In both items (1) and (2), the admissible frame is
regular if and only if the corresponding potential pair is regular. Moreover,
with notation as in Definitions \ref{admissibleframedef} and  \ref{potentialdefn},
we have $\textup{sign}[A_1]_{21} = \textup{sign}[\psi^X_1]_{21}$ and
$\textup{sign}[A_{-1}]_{12} = \textup{sign}[\psi^Y_{-1}]_{12}$. In fact, we have 
\beq \label{mcpotexpression}
\hat F^{-1}\dd \hat F=\lambda\psi^X_1\dd x+\alpha_0+\lambda^{-1}\hat H_+\big|_{\lambda=0}\psi^Y_{-1}\hat H_+^{-1}\big|_{\lambda=0}\dd y,
\eeq
where $\alpha_0$ is constant in $\lambda$. \\
\end{enumerate}
\end{theorem}




\section{Frontals and fronts}
\noindent
For the rest of this article we will be interested in timelike CMC surfaces with singularities. An appropriate class of generalized surface is a frontal.
Here we briefly outline some definitions and results from \cite{krsuy} and \cite{fsuy}.\\

\noindent
Let $M$ be a $2$-dimensional manifold.  A map $f: M \to \E^3$, into the three-dimensional
Euclidean space, is called a \emph{frontal} if, on a neighbourhood $U$ of
any point of $M$,
 there exists a unit vector field $\enormal: U \to \SSS^2$, well-defined up to sign, such that $\enormal$ is perpendicular to $\dd f(TM)$ in $\E^3$.   The map
$L = (f, \, [\enormal]): M \to \E^3 \times \real P^2$ is called a 
\emph{Legendrian lift} of $f$. If $L$ is an immersion, then $f$ is called a \emph{front}.  
A point $p \in M$ where a frontal $f$ is not an immersion is called 
a \emph{singular point} of $f$.\\

\noindent
Suppose that the restriction of a frontal $f$, to some open dense set,
  is an immersion, and some Legendrian lift $L$ of  
$f$ is given. Then, around any point in $M$,
 there exists a smooth function $\chi$,
given in local coordinates $(x,y)$
by the Euclidean inner product $\chi=\ip{(f_x\times f_y)}{\enormal}_E$, such that
\bdm
f_x \times f_y = \chi \enormal.
\edm
In this situation, a singular point $p$ is called \emph{non-degenerate}
if $\dd \chi$ does not vanish there, and  the frontal $f$ is called non-degenerate if every singular point is non-degenerate. 
The set of singular points is locally given as the zero set of $\chi$,
and is a smooth curve (in the coordinate domain) around non-degenerate points.
At such a point, $p$, there is a well-defined
direction, that is a non-zero vector $\eta \in T_pM$, unique up to scale,
 such that $\dd f(\eta) = 0$, called the \emph{null direction}.  \\

\subsection{The Euclidean unit normal}   \label{enormalsect}
In order to use the framework above, we need the Euclidean unit
normal to a CMC surface.
The orthonormal basis, $e_0$, $e_1$, $e_2$ for $\LL^3$ satisfy the commutation
relations $[e_0, e_1] = 2 e_2$, $[e_1, e_2] = -2 e_0$ and $[e_2, e_0] = 2 e_1$.
Defining the standard cross product on the vector space $\R^3 = \LL^3$,
with $e_0 \times e_1 = e_2$, $e_1 \times e_2 = e_0$ and $e_2 \times e_0 = e_1$,
we have the formula:
\bdm
A \times B = - \frac{1}{2}\Ad_{e_0} [A, B].\\
\edm

\noindent
From Proposition \ref{characterizationprop}, the coordinate frame for a regular timelike surface associated
to an admissible frame  is
$f_x = \varepsilon_1 e^{\omega/2} \Ad_{F_C} (e_0+e_1)/2$,
$f_y = \varepsilon_2 e^{\omega/2} \Ad_{F_C} (-e_0+e_1)/2$ and
$N = \Ad_{F_C} e_2 = \Ad_F e_2.$
We can use these to compute the cross product
\beq  \label{fxcrossfy}
f_x \times f_y = -\frac{e^\omega}{2} \varepsilon \Ad_{e_0} \Ad_{F_C}(e_2) = -\frac{e^\omega}{2} \varepsilon \Ad_{e_0} N,
\eeq
where $\varepsilon=\varepsilon_1\varepsilon_2$. This formula is valid provided 
the surface is regular, that is, $c_1 \neq 0 \neq b_2$.   However, 
the formula $N= \Ad_F e_2$ is valid everywhere, and gives a smooth vector field
on $M$.  Therefore, we define the \emph{Euclidean unit normal $\enormal$} to $f$ 
to be 
\beq  \label{enormaldef}
\enormal := \frac{\Ad_{e_0} \Ad_F (e_2)}{|| \Ad_F (e_2)||},  
\eeq
where $|| \cdot ||$ is the standard Euclidean norm on the
vector space $\R^3$ representing $\LL^3$. 
At points where the surface is regular, we have
\bdm
\enormal = -\varepsilon \frac{f_x \times f_y}{|| f_x \times f_y ||}.
\edm
For other values of $\lambda \in \real \setminus \{0\}$ one defines the analogue
$\enormal^\lambda$ for $f^\lambda$, by
replacing $F$ with $F^\lambda$.\\

\section{Singularities of class I: On the big cell}  \label{classIsect}
\noindent
We now want to study the singularities occurring on a timelike CMC surface produced from a semi-regular potential pair $(\psi^X, \psi^Y)$, as in  Theorem \ref{dpwthm}.  \\

\noindent
We first consider the case that the map $\hat \Phi = \hat X^{-1} \hat Y$ takes values in the big cell $\mathcal{B}_L$.
In this case, the formula (\ref{enormaldef}) for $\enormal$ shows  that the Euclidean unit normal is never lightlike, regardless of whether 
the surface is immersed or not.  Conversely, we will later show that, for singularities occurring at the big cell boundary, the Euclidean normal is \emph{always} lightlike;
 this is the geometric difference between the two cases, which we will call \emph{class I}
and \emph{class II} respectively. We now consider the generic singularity of the
first case.\\

\noindent
Given a potential pair, $(\psi^X, \psi^Y) = (O_\infty(1) + \psi^X_1 \lambda\, ,\, \psi^Y_{-1}\lambda^{-1} + O(1))$, we can write
\beqas
\psi^X_{1} :=  \bbar 0 & \alpha \\ \beta & 0 \ebar, \quad
\psi^Y_{-1} :=  \bbar 0 & \gamma \\ \delta & 0 \ebar,
\eeqas
where $\alpha$ and $\beta$ are real and depend on $x$ only, and $\gamma$ and $\delta$ are
real and depend on $y$ only. From the converse part of Theorem \ref{dpwthm},
we see that these functions are otherwise completely arbitrary. 
If $\hat \Phi$ takes values in $\mathcal{B}_L$,  the surface
$f = \frac{1}{2H} \sym_1(\hat F)$ will have singularities when either of $\beta$ or $\gamma$ are zero, and is immersed otherwise.  Thus, for a semi-regular potential, for which at 
most one of these is allowed to vanish, and this to first order,
 a singularity occurs at $z_0 = (x_0,y_0)$ if and only if
 $\beta(x_0)=0$, $\frac{\dd \beta}{\dd x} (x_0)\neq 0$ and 
  $\gamma(y_0)\neq 0$
  (or the analogue, switching $y$ with $x$ and $\beta$ with $\gamma$).
For the \emph{generic} case, the  function $\alpha$ is
 also non-zero at $z_0$.\\

\begin{figure}[here]  
\begin{center}
\includegraphics[height=40mm]{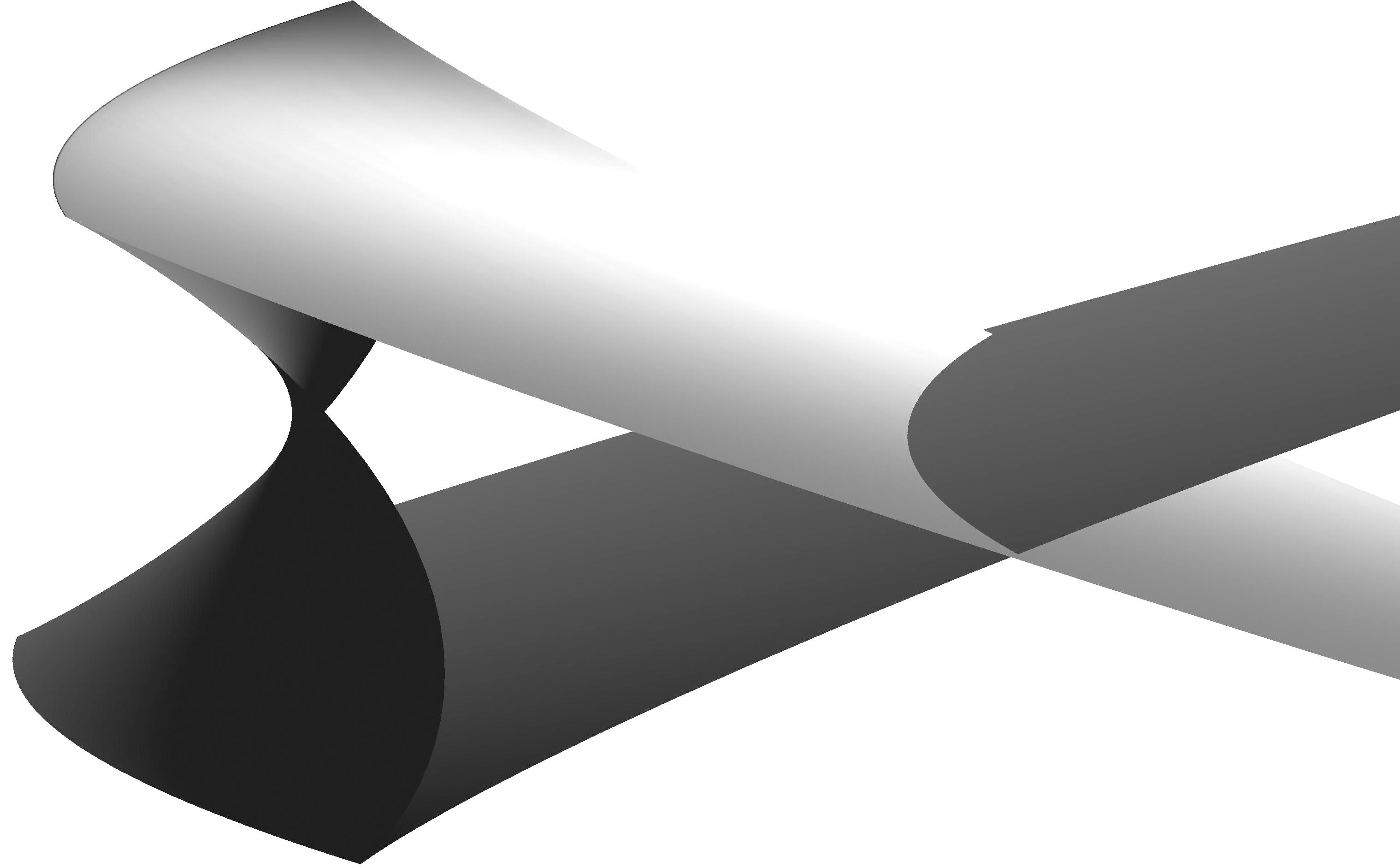} 
\end{center}
\caption{A numerical plot of a timelike CMC surface with  cuspidal edges along
 the two coordinate lines $x=\pm 1$, produced by a pair of potentials
with $\beta = (x-1)(x+1)$ and $\alpha = \gamma = \delta = 1$.}
\label{figure2}
\end{figure}

\noindent
We quote a characterization of the cuspidal edge from Proposition 1.3 in \cite{krsuy}:
\begin{lemma} \label{cuspedgelemma}
 Let $f$ be a front and $p$ a non-degenerate singular point.
 The image of  $f$ in a neighbourhood of
 $p$ is diffeomorphic to a
cuspidal edge if and only if  the null direction
$\eta(p)$ is transverse to the
singular curve.
\end{lemma}

\noindent
Now we can describe the  generic singularities of a semi-regular surface on the big cell:
\begin{proposition} 
If the map $\hat \Phi = \hat X^{-1} \hat Y$ corresponding to a semi-regular
potential pair takes values in $\mathcal{B}_L$, then
a generic singularity of the surface $f := \frac{1}{2H} \sym_1(\hat F)$
is a cuspidal edge.
\end{proposition}
\begin{proof}
Clearly $(f, \enormal)$ defines  a frontal, where $\enormal$ is defined by 
equation (\ref{enormaldef}).  Assume now that, at $z_0=(0,0)$,
 we have
 $\beta=0$, $\frac{\dd \beta}{\dd x} \neq 0$ and 
  $\alpha \neq 0$.
 Writing $\enormal = \mu \Ad_{e_0} \Ad_F e_2$, with
 $\mu = || \Ad_F e_2 ||^{-1}$, and examining the off-diagonal 
 components in
\bdm
 \Ad_{F^{-1}} \Ad_{e_0}(\dd \enormal) = \mu [ F^{-1} \dd F, e_2] + \dd \mu \, e_2,
\edm
shows that $\enormal$ is an immersion.
Hence the map $(f, \enormal)$ is regular, and $f$ is a front. \\

\noindent
To show that the singular point is non-degenerate we need to show that 
$\dd \chi(z_0) \neq 0$, where 
\beqas
\chi &=& \ip{(f_x \times f_y)}{\enormal}_E = -\dfrac{e^\omega}{2} \varepsilon \ipb{\Ad_{e_0} \Ad_F (e_2)}{\dfrac{\Ad_{e_0} \Ad_F (e_2)}{||\Ad_F(e_2)||}}_E\\
&=& -\dfrac{\varepsilon e^{\omega}}{2} || \Ad_F(e_2)|| \\
&=& \dfrac{2b_2 c_1}{H^2}|| \Ad_F(e_2)||,
\eeqas
in the notation of Proposition \ref{characterizationprop}. 
Now using the expression (\ref{mcpotexpression}) for $\hat F^{-1} \dd \hat F$,
we observe that $c_1 = \beta$.
Hence we obtain, at $(0,0)$,
\bdm
\dfrac{\partial \chi}{\partial x} = 
 \dfrac{\dd \beta }{\dd x} \dfrac{2 b_2}{H^2} ||\Ad_F(e_2)||.
\edm
This is non-zero, since we assumed that $\frac{\dd \beta }{\dd x} \neq 0$,
and, as mentioned in Theorem \ref{dpwthm},
 $b_2$ vanishes if and only if $\gamma$ vanishes.
 Hence $\dd \chi$ does not vanish at $(0,0)$.\\

\noindent
 According to  Lemma
\ref{cuspedgelemma}, we need to show that
the singular curve is transverse to the null direction.
 In a neighbourhood of $(0,0)$, the singular
curve is given by the equation $x=0$, that is, it is
tangent to  $\partial _y$. 
Finally, since $f_x = 0$ and $f_y\neq 0$ at $(0,0)$, the null
direction at this point is $\eta(0) = \partial _x$.\\ 
\end{proof}

\section{Singularities of class II:  At the big cell boundary}
\noindent
We now turn to singularities that occur due to the failure of 
the loop group splitting at the boundary of the big cell. We again assume that the 
potentials corresponding to the surface are  semi-regular at the points in question.\\

\noindent
We need the Birkhoff decomposition of the whole group $\G^\C$:
\begin{theorem}\cite{PreS,brs} \label{birkhoff2} Every element $\gamma \in\G^\C$ which is not in the left big cell $\mathcal{B}_L$ can be written as a product
\bdm
\gamma = \gamma_- \, \omega \, \gamma_+,
\edm
where $\gamma_\pm\in\G^{\C\pm}$ and the middle term $\omega$ is uniquely determined by $\gamma$ and has the form
\bdm
\omega = \bbar \lambda^{2 n} &  0 \\ 0& \lambda^{-2n} \ebar, \,\,
n \in {\mathbb Z\setminus \{ 0 \}},\,\,\,
\textup{or} \,\,\, 
\omega = \bbar 0 & \lambda^{2n+1} \\ - \lambda^{-(2n+1)} & 0 \ebar,
\,\, n \in {\mathbb Z}.
\edm
The same statement holds replacing $\mathcal{B}_L$ with $\mathcal{B}_R$ and interchanging $\gamma_-$ and $\gamma_+$.\\
\end{theorem}

\noindent
We write 
\beqas
\omega_k = \bbar \lambda^{ k} & 0 \\ 0 & \lambda^{- k} \ebar\ \ (\text{$k$ even}),\quad \omega_k = \bbar 0 & \lambda^{k} \\ -\lambda^{-k} & 0 \ebar\ \ (\text{$k$ odd}),
\eeqas
and 
$$
\PP_L^k=\{\gamma_-\,\omega_k\,\gamma_+\ |\ \gamma_\pm\in\G^{\C\pm}\}\qquad (k\in\Z).
$$
We note that 
$$
\Ad_{e_0}(\PP_L^k)=\PP_L^{-k}\ \ (\text{$k$ odd}),\qquad \Ad_{e_1}(\PP_L^k)=\PP_L^{-k}\ \ (\text{$k$ even}).\\
$$

\subsection{Behaviour of the surface at $\PP_L^{\pm 1}$ and $\PP_L^{\pm2}$.} 
The behaviour of the surface and its admissible frame at the smaller cells $\PP_L^{\pm1}$ and $\PP_L^{\pm2}$ is explained in the following result. 

\begin{theorem}  \label{firstsmallcellthm}
Let $\hat X: I_x \to \G$ and $\hat Y: I_y \to \G$ be obtained from a real analytic semi-regular potential pair as in Theorem \ref{dpwthm}. Set $M=I_x\times I_y$ and $\hat \Phi = \hat X^{-1} \hat Y$. Suppose that $M^\circ=\hat \Phi^{-1}(\B_L)$ is non-empty. If for some $z_0 = (x_0,y_0) \in M$, $\hat \Phi (z_0) = \omega_j$ for $j=\pm 1$ or $\pm 2$, then 
\begin{enumerate}
\item \label{thmitem1}
$M^\circ$ is open and dense in $M$;\\
\item \label{thmitem2}
 if $j = 1$, then the surface $f^\lambda : M^\circ \to \LL^3$ obtained as $f^\lambda = \frac{1}{2H}\sym_\lambda (\hat F)$, for $\lambda \in \real \setminus \{0\}$, where $\hat F=\hat Y \hat H_+^{-1}$ as in Theorem \ref{dpwthm}, extends continuously to $z_0$, is real analytic in a neighbourhood of $z_0$, but is not immersed at $z_0$; moreover, the Euclidean unit normal is lightlike at $z_0$;\\
\item \label{thmitem3}
if $j=-1$,  or $2$, then $\lim_{z\to z_0} \|f^\lambda \| = \infty$, where the limit is over values $z \in M^\circ$; \\
\item \label{thmitem4}
if $j=-2$ then $\lim_{z\to z_0} \|f^\lambda \|$ may be finite or infinite, depending on the 
sequence $z \to z_0$, but  $f$ is not an immersed timelike surface at $z_0$.\\
\end{enumerate}
\end{theorem}

\begin{remark} In the statement of the theorem, the assumption that the potential pair is real analytic is only used in item (1). By adding (1) as an assumption, (2), (3) and (4) still remain true (replacing real analytic with smooth in (2)) if the potential pair is only assumed smooth. 
\end{remark}

\noindent
To prove the theorem we need two lemmas, both of which are verified by simple algebra.  
\pagebreak
\begin{lemma}\label{switchlemma} Let $ H_- = \bbar a & b \\ c & d \ebar \in \G_-$.  
\begin{enumerate}
\item If $c_{-1}\neq 0$, then $\omega_1H_-\in\B_L$ has a left Birkhoff decomposition
\bdm
\omega_1 \hat H_- = \bbar \lambda c & \lambda d - u_0 \lambda^2 c \\ -\lambda^{-1} a & u_0 a - \lambda^{-1} b \ebar \bbar 1 & u_0 \lambda \\ 0 & 1 \ebar,
\edm
where $u_0 = d_0/c_{-1}$. 
\item If $c_{-1} = 0$, then $\omega_1 \hat H_-\in\PP_L^1$ has a left Birkhoff decomposition
\bdm
\omega_1 \hat H_- = \bbar d & -\lambda^2 c \\ -\lambda^{-2}b & a \ebar \omega_1.
\edm
\item If $b_{-1}\neq 0$ then $\omega_{-1}\hat H_-\in\B_L$ has a left Birkhoff factorization 
\bdm
\omega_{-1} \hat H_-=\bbar \lambda^{-1}c-v_0d & \lambda^{-1}d \\ -\lambda a+\lambda^2bv_0 & -\lambda b \ebar\bbar 1 & 0 \\ \lambda v_0 & 1 \ebar, 
\edm
where $v_0=a_0/b_{-1}$. 
\item If $b_{-1}=0$, then $\omega_{-1}\hat H_-\in\PP_L^{-1}$ has a left Birkhoff decomposition 
$$
\omega_{-1}\hat H_-=\bbar d & -\lambda^{-2}c \\ -\lambda^2b & a \ebar \omega_{-1}.\\
$$
\end{enumerate}
\end{lemma}

\begin{lemma}\label{switchlemma2} Let $\hat H_-=\bbar a & b \\ c & d \ebar\in\G_-$. Suppose that $b_{-1} \neq 0$ and $a_{-2}b_{-1} - a_0 b_{-3} \neq 0$.
Then $\omega_2\hat H_-\in\B_L$ has a left Birkhoff factorization $\omega_2 \hat H_- = G_-G_+\,$, where 
\beqas
G_+ = \bbar 1 +  \dfrac{a_0 b_{-1}}{a_{-2}b_{-1} - a_0 b_{-3}} \lambda^2 & \dfrac{b_{-1}^2}{a_{-2}b_{-1} - a_0 b_{-3}} \lambda \vspace{1ex} \\
   \dfrac{a_0}{b_{-1}} \lambda & 1 \ebar.\\
\eeqas
\end{lemma}

\begin{proof}[Proof of Theorem \ref{firstsmallcellthm}]
\emph{Item (\ref{thmitem1}):} The big cell $\mathcal{B}_L$ is the complement of the zero set of a holomorphic section in a line bundle over the complex loop group \cite{DorPW}.  Thus $M^\circ$ is the complement of the zero set of a real analytic section of the pull-back of this bundle by $\hat \Phi$. Since we have assumed that this set is non-empty, it must be open and dense. \\

\noindent
\emph{Item (\ref{thmitem2}):}
On $M^\circ$, we perform a left normalized Birkhoff decomposition $\hat\Phi=\hat H_-\hat H_+$. Since $\omega_1^{-1}\hat\Phi$ takes values in $\B_L$ in a neighbourhood $U$ of $z_0$, we perform a left normalized Birkhoff decomposition $\omega_1^{-1} \hat \Phi=\hat G_-\hat G_+$ in $U$. Thus, in $M^\circ\cap U$, we have $\hat H_-\hat H_+=\omega_1\hat G_-G_+$. Applying Lemma \ref{switchlemma} to 
$$
\hat G_-=\bbar a & b \\ c & d \ebar,
$$
gives 
$$
\hat H_- \hat H_+=\omega_1 \hat G_- \hat G_+=\bbar \lambda c & \lambda d- \lambda^2 \dfrac{c}{c_{-1}} \\ -\lambda^{-1}a & \dfrac{a}{c_{-1}} -\lambda^{-1}b \ebar \bbar 1 & \dfrac{1}{c_{-1}}\lambda \\ 0 & 1 \ebar \hat G_+.
$$
By uniqueness of the normalized Birkhoff factorization, we see that 
$$
\hat H_-= \omega_1 \hat G_- \hat U_+^{-1} D^{-1}, \quad \quad
\hat H_+ = D \hat U_+ \hat G_+,
$$
where 
$$
\hat U_+=\bbar 1 & \dfrac{1}{c_{-1}}\lambda \\ 0 & 1 \ebar,\qquad D=\bbar c_{-1} & 0 \\ 0 & \dfrac{1}{c_{-1}} \ebar.
$$
Then 
\beq   \label{2sym}
\sym_\lambda (\hat F)=\sym_\lambda (\hat Y \hat H_+^{-1})= \sym_\lambda (\hat Y\hat G_+^{-1}\hat U_+^{-1}D^{-1})=\sym_\lambda(\hat Y\hat G_+^{-1}),
\eeq
 because  $\sym_\lambda$ is invariant under postmultiplication by matrices of the form $\hat U_+$ and $D$ . Setting $\tilde F:= \hat Y \hat G_+^{-1}$, which  is well defined and analytic in a neighbourhood of $z_0$, we have just shown that
  $\sym_\lambda(\tilde F) = \sym_\lambda(\hat F)$ on the intersection of their domains of definition. Hence  $f^\lambda$ is well defined and analytic around $z_0$.\\

\noindent
To see that $f^\lambda$ is not immersed at $z_0$, we have by Theorem \ref{dpwthm}, 
$$
\hat F^{-1}\dd\hat F=\lambda\psi^X_1 \dd x+\alpha_0+\hat H_+\big|_{\lambda=0} \, \psi^Y_{-1}\hat H_+^{-1} \big|_{\lambda=0} \, \lambda^{-1} \dd y.
$$
We can write $\hat G_+=\textup{diag}(A_0 \, , \, A_0^{-1}) + O(\lambda)$.  Then
$\hat H_+ = \textup{diag}( c_{-1}A_0 \, , \,   (c_{-1}A_0)^{-1}) + O(\lambda)$. 
Hence, if 
\bdm
\psi^X_1=\begin{pmatrix} 0 & \beta \\ \gamma & 0 \end{pmatrix},\quad \psi^Y_{-1}=\begin{pmatrix} 0 & \delta \\ \sigma & 0 \end{pmatrix},
\edm
then
$$
\hat H_+\big|_{\lambda=0} \psi^Y_{-1}\hat H_+^{-1} \big|_{\lambda=0}=\begin{pmatrix} 0 & (c_{-1}A_0)^2\delta \\ \dfrac{\sigma}{(c_{-1}A_0)^2} & 0 \end{pmatrix}.
$$
As the potential is semi-regular, $\gamma$ and $\delta$ do not vanish simultaneously,
and their zeros are of first order, and therefore isolated.  At points where
these functions are non-zero, we 
set, as in Proposition \ref{characterizationprop}, 
$$
\rho=\bigg|\dfrac{(c_{-1}A_0)^2\delta}{\gamma}\bigg|^{1/4}\ \text{  and  }\ T=\begin{pmatrix} \rho & 0 \\ 0 & \rho^{-1} \end{pmatrix}
$$
and $\hat F_C=\hat FT$. We have $\hat F=\hat Y\hat H_+^{-1}=\hat Y\hat G_+^{-1}\hat U_+^{-1}D^{-1}$, hence, by the formulae (\ref{signedcoordframe}),
\beqas
f^\lambda _x &=& \dfrac{\lambda  \gamma\rho^2}{H}\Ad_{F^\lambda _C}(e_0+e_1)= 
\dfrac{\lambda  \gamma}{H}\Ad_F^\lambda (e_0+e_1)\\
&=&
 \dfrac{ \lambda \gamma}{H}\Ad_{Y^\lambda (G^\lambda)_+^{-1}}((c_1^2+1)e_0+(c_{-1}^2-1)e_1-2c_{-1}e_2).
\eeqas
The last expression is well-defined and smooth, even at a point where $\gamma$ or $\delta$ vanishes,
and therefore valid everywhere. Similarly,
$$
f^\lambda _y=  \dfrac{(c_{-1}A_0)^2\delta\rho^{-2}}{\lambda H}\Ad_{F ^\lambda _C}(e_0 - e_1)
  = \dfrac{A_0^2\delta}{\lambda H}\Ad_{Y^\lambda(G^\lambda _+)^{-1}}(e_0 - e_1).
$$
As $A_0\to 1$ and $c_{-1}\to 0$ when $z\to z_0$, we have
\beq \label{limitingderivatives}
\begin{split}
f^\lambda _x(z_0)=  \frac{ \lambda\gamma}{H}\Ad_{Y^\lambda(G^\lambda_+)^{-1}}(e_0 - e_1),\\
f^\lambda _y(z_0)=  \frac{\delta}{\lambda H}\Ad_{Y^\lambda(G_+^\lambda)^{-1}}(e_0 -e_1).
\end{split}
\eeq
Thus we have proved that $f^\lambda$ is not immersed at $z_0$. \\

\noindent 
To see that the Euclidean normal is lightlike, 
see the explicit formula given below in Lemma \ref{enormallemma}.
Alternatively,
one can first show that the Minkowski unit normal  $N^\lambda$ blows up (and therefore is asymptotically lightlike) by considering the surface
$f^\lambda _P$ obtained from $\Ad_{e_1}\hat F$, which (one computes from the Sym formula) 
is the parallel surface to $f^\lambda$. Since 
$\Ad_{e_1} \hat \Phi (z_0) = \omega_{-1}$, we show below that $f^\lambda _P$ blows up, and 
therefore so does $N^\lambda$.  Hence the formula $(\ref{enormaldef})$ for the Euclidean
normal shows that $\enormal^\lambda$ is lightlike at $z_0$. 
\\

\noindent
\emph{Item (\ref{thmitem3}):} As in item (\ref{thmitem2}), we write $\omega_{-1}^{-1} \hat \Phi=\hat G_-\hat G_+$ in a neighbourhood $U$ of $z_0$, and $\hat \Phi=H_-H_+$ in $U\cap M^\circ$. Again denoting the components of $\hat G_-$ by 
$a$, $b$, $c$ and $d$, 
 Lemma \ref{switchlemma} says that $\hat H_+=D\hat U_+\hat G_+$, where now 
$$
D=\bbar -b_{-1}^{-1} & 0 \\ 0 & -b_{-1} \ebar,\qquad \hat U_+=\bbar 1 & 0 \\ \lambda b_{-1}^{-1} & 1 \ebar.
$$
Hence 
$$
2Hf^\lambda =\sym_\lambda(\hat Y\hat G_+^{-1}\hat U_+^{-1})
  =\Ad_{Y^\lambda(G^\lambda_+)^{-1}}(\sym_1(\hat U_+^{-1}))+\sym_1(\hat Y\hat G_+^{-1}).
$$
Since $\hat Y\hat G_+^{-1}$ are well defined and real analytic in $U$, the second term is finite in $U$, while the first term is given by 
$$
\Ad_{Y^\lambda(G^\lambda _+)^{-1}}(\sym_1(\hat U_+^{-1}))=-2b_{-1}^{-1}\Ad_{Y^\lambda(G^\lambda_+)^{-1}}(e_0+e_1)-\Ad_{Y^\lambda(G^\lambda_+)^{-1}}e_2.
$$
The second term is finite in $U$, while the first term goes to infinity as $z\to z_0$, since $b_{-1}\to 0$ in this case. This proves item (\ref{thmitem3}) for $j=-1$. \\

\noindent
If $\hat\Phi(z_0)=\omega_2$, we proceed as in the case just described, choosing a suitable neighbourhood $U$ of $z_0$ and write $\omega_2^{-1}\hat\Phi=\omega_2^{-1}\hat H_-\hat H_+=\hat G_-\hat G_+$ on $U\cap M^\circ$. Using the 
same notation for the components of 
$\hat G_-$,
we have by Lemma \ref{switchlemma2}, $\hat H_+=D\hat U_+\hat G_+$, where $D$ is a diagonal matrix constant in $\lambda$, and 
$$
\hat U_+=\bbar 1 +  \dfrac{b_{-1}}{a_{-2}b_{-1} - b_{-3}} \lambda^2 & \dfrac{b_{-1}^2}{a_{-2}b_{-1} - b_{-3}} \lambda \\ \dfrac{1}{b_{-1}} \lambda & 1 \ebar.
$$
We have 
$$
2Hf^\lambda=\sym_1(\hat Y\hat G_+^{-1}\hat U_+^{-1}D^{-1})=\Ad_{Y^\lambda(G^\lambda_+)^{-1}}\left(\sym_1(\hat U_+^{-1}))+\sym_1(\hat Y\hat G_+^{-1})\right).
$$
The second term is finite, while the first term is given by 
$$
\Ad_{Y^\lambda(G^\lambda_+)^{-1}}\sym_1(\hat U_+^{-1}))=-2b_{-1}^{-1}\Ad_{Y^\lambda(G^\lambda_+)^{-1}}(e_0+e_1)-\Ad_{Y^\lambda(G^\lambda_+)^{-1}}e_2,
$$
and the conclusion follows as in the case when $j=-1$.  \\

\emph{Item (\ref{thmitem4}):}
The case when $j=-2$ can be computed in an analogous way to $j=2$. Instead of the equation above,
one is led to:
$$
\Ad_{Y^\lambda(G^\lambda_+)^{-1}}\sym_1(\hat U_+^{-1}))= 
\Ad_{Y^\lambda(G^\lambda_+)^{-1}} \bbar  4\Delta \lambda^2 & 4 \Delta c_{-1}^{-1} \lambda^3 \\
     -4 c_{-1} \Delta \lambda & -4 \Delta \lambda^2 \ebar 
   -\Ad_{Y^\lambda(G^\lambda_+)^{-1}}e_2,
$$
where $\Delta = c_{-1}/(d_{-2} c_{-1} - c_{-3})$, and the functions $c_{-1}$, $d_{-2}$ and 
$c_{-3}$ all approach zero as $z\to z_0$.  Since it is possible to choose sequences such that
the right hand side of the above equation is either finite or infinite as $z\to z_0$, we can say nothing about this limit.  If the limit is finite, we can deduce that the map $f$ is not an
immersion as follows:  by the same argument described above for $j=1$, namely considering
the surface $f_P^\lambda$, which blows up, since $f_P^\lambda(z_0) \in \mathcal{P}_L^{2}$,
one deduces that the Minkowski normal must be lightlike at $z_0$. This cannot happen on
an immersed timelike surface.\\
\end{proof}

\noindent
Note that generic singularities should not occur at points in $\mathcal{P}_L^j$ for $|j|>1$,
because the codimension of the small cells in the loop group increases with $|j|$.
In view of the previous theorem, and with the aim of studying surfaces with finite, generic singularities,  we make the following definition:  
\begin{definition} A \emph{generalized timelike CMC $H$ surface} is a smooth map
$f: \Sigma \to \LL^3$,  from an oriented surface $\Sigma$,
such that,
at every point $z_0$ in $\Sigma$,  the following holds: there exists a neighbourhood $U$ 
of $z_0$ such that the restriction $f\big|_U$ can be represented by a 
semi-regular
 potential pair $(\psi^X, \psi^Y)$, where the corresponding map $\hat \Phi = \hat X^{-1} \hat Y$
maps $U$ into $\mathcal{B}_L \cup \mathcal{P}_L^1$, and where $\hat \Phi^{-1}(\mathcal{B}_L)$ is open and dense in $U$. If the potential pair is regular,
the surface is called \emph{weakly regular}.
\end{definition}
Note that if  $f$ is weakly regular, that is, represented by a \emph{regular} potential 
pair at each point, then $f$ is immersed precisely at those points 
for which the corresponding map $\hat \Phi$ maps into the big cell $\mathcal{B}_L$.
In other words, there is a well defined open dense set $\Sigma^\circ$ on which 
$f$ is an immersion and $f$ will have singularities precisely at points which map into
$\mathcal{P}_L^1$.\\

\section{Prescribing class II singularities 
of non-characteristic type}   \label{noncharsect}
\noindent
We have seen that the Euclidean unit normal $\enormal$ is well defined at a singularity
occurring on the big cell.  Below we will show that this is also the case for those at
the big cell boundary.  Then we have seen in the previous sections that singularities 
in the two cases can be distinguished by the property that $\enormal$ is not lightlike in
the first case, and is lightlike in the second case, 
which we have already named class I and class II respectively. \\
\noindent
Constructing surfaces with a prescribed singular curve of the class I is
 simple: it is a matter of solving
the geometric Cauchy problem for the characteristic case (see \cite{dbms1}),
which has infinitely many solutions, and choosing the second potential to be
non-regular at the point in question.    Therefore, we henceforth discuss
 only singularities of class II.

\subsection{Singular potentials}  \label{singpotsection}
Assume now that we are at a non-degenerate singular point $p = f(0,0)$,
so that the pre-image of the singular set in a neighbourhood of $p$ is given
by some curve $\Gamma: (\alpha , \beta) \to M$.  Assume
that $\Gamma$ is never parallel to a lightlike coordinate line $y= \textup{constant}$
or $x=\textup{constant}$, which means that the singular curve is non-characteristic
for the associated PDE. The characteristic case will be discussed in the next section.\\

 \noindent
 With the non-characteristic assumption, one
 can express $\Gamma$ as a graph, 
$y=h(x)$, with $h^\prime(x)$ non-vanishing, and, after a
change of coordinates $(\tilde x, \tilde y) = (h(x), y)$, which are still
lightlike coordinates for the regular part of the surface, one can even assume
that $\Gamma$ is given by $y=x$, which is to say $u=0$ in the coordinates
$$
u = \frac{1}{2}\left(x-y\right), \hspace{1cm} v= \frac{1}{2}\left(x+y\right).
$$
Note that we could distinguish the cases $h^\prime > 0$ and $h^\prime <0$,
which corresponds
to the curve being spacelike/timelike in the coordinate domain, but nothing
fundamentally new is gained by doing this.\\

\noindent
The issues discussed below are local in nature, and therefore
 we assume that our parameter space is a square, $M = J \times J \subset \real^2$, where $J$ is an open interval containing $0$.
In these coordinates, along the line $y=x=v$ we have, by definition of 
$\mathcal{P}_L^1$,
\bdm
\hat \Phi(v) = \hat X^{-1}(v) \hat Y(v)= \hat G_-(v) \, \omega_1 \, \hat G_+(v),
\edm
with $\hat G_-(v) \in \gggg^-$ and $\hat G_+(v) \in \gggg^+$.
It is also easy to show, using the expressions in Lemma \ref{switchlemma}, that
if $\hat \Phi$ is smooth then $\hat G_-$ and $\hat G_+$ can also be chosen to be smooth.
We can replace the map $\hat X(x)$ by $\hat X(x) \hat G_-(x)$,
 and $\hat Y(y)$ by $\hat Y(y) \hat G_+^{-1}(y)$, which correspond to the 
standard potential pair 
\beqas
\psi^X= \hat G_-^{-1}(\hat X^{-1} \dd \hat X) \hat G_- + \hat G_-^{-1}\dd \hat G_-,\\
\psi^Y = \hat G_+(\hat Y^{-1} \dd \hat Y) \hat G_+^{-1} + \hat G_+ \dd \hat G_+^{-1},
\eeqas
and it is simple to check that the surface constructed from these potentials is 
 the same as the original surface. Thus one can, in fact, assume that 
\bdm
\hat \Phi(v) = \hat X^{-1}(v) \hat Y(v)=  \omega_1.\\
\edm

\noindent
Finally, choosing a normalization point $z_0= (0,0)$ on the singular set, one can also 
assume that 
\bdm
\hat X(z_0) = \omega_1^{-1}, \quad \hat Y(z_0) = I.
\edm
This is achieved by premultiplying both $\hat Y$ and $\hat X$ by $\hat Y^{-1}(z_0)$.
This leaves $\hat \Phi$ unchanged, and alters the surface 
$f = (1/2H)\sym_1(\hat Y \hat H_+^{-1})$ of Theorem \ref{dpwthm} only by an isometry
consisting of conjugation  by $Y^{-1}(z_0)$ plus a translation.\\

\noindent
 As shown by equation (\ref{2sym}) in
  Theorem \ref{firstsmallcellthm}, we can equivalently
consider the map $\tilde \Phi := \omega_1^{-1} \hat \Phi$, which is the same as replacing $\hat X$ by
$\tilde X := \hat X \omega_1$.  Therefore, we first look at the Maurer-Cartan form
of $\tilde X$, given that $\hat X^{-1} \dd \hat X$ is a standard potential of the form:
\bdm
\bbar \alpha_0 +O_\infty(\lambda^{-2}) & \beta_1 \lambda + \beta_{-1} \lambda^{-1} \beta_{-3} \lambda^{-3} + O_\infty(\lambda^{-5}) \\ \gamma_1 \lambda + \gamma_{-1} \lambda^{-1} + \gamma_{-3} \lambda^{-3} + O_\infty(\lambda^{-5}) & -\alpha_0 + O_\infty(\lambda^{-2}) \ebar \dd x.
\edm
Then
\beqas 
\tilde X^{-1} \dd \tilde X &=& \omega_1^{-1} (\hat X^{-1} \dd \hat X^{-1}) \omega_1 \\
&=& \bbar -\alpha_0  & -\gamma_1 \lambda^3 - \gamma_{-1} \lambda -\gamma_{-3} \lambda^{-1}  \\ -\beta_1 \lambda^{-1}  & \alpha_0 \ebar \dd x + O_\infty(\lambda^{-2}).
\eeqas
Now we observe that, since $\tilde X^{-1} (v) Y(v) = I$ for all $v$,
we actually have 
\bdm
\tilde X(v) = Y(v)
\edm
for all $v$.  It follows that, along $y=x$, we have $\tilde X^{-1} \dd \tilde X = \psi^Y$,
which was assumed to be a standard potential, and so all the terms of order $-2$ or lower
in $\lambda$ are zero. \\

\begin{definition}  \label{singpotdef}
A \emph{singular potential} on an open interval $J \subset \real$,
 is a  $\textup{Lie}(\mathcal{G})$-valued 1-form
$\tilde \psi$ on $J$ which has the
Fourier expansion  in $\lambda$:
\bdm
\tilde \psi =\bbar -\alpha_0  & -\gamma_1 \lambda^3 - \gamma_{-1} \lambda -\gamma_{-3} \lambda^{-1}  \\ -\beta_1 \lambda^{-1}  & \alpha_0 \ebar \dd v =: A(v) \dd v.
\edm
Any zeros of $\gamma_1$ and $\gamma_{-3}$ are of at most first order.
The potential is \emph{regular} at points where $\gamma_1$ and $\gamma_{-3}$ do  not vanish. 
The potential is \emph{non-degenerate} at points where $\beta_1$ does not vanish.
\end{definition}

\noindent
We have seen by the above argument that a timelike CMC surface that has
a non-degenerate singular point gives us a singular potential $\tilde \psi(v) = A(v) \dd v$, and moreover is
reconstructed, up to an isometry of the ambient space,
by integrating $\tilde X^{-1} \dd \tilde X (x) =  A(x) \dd x$
and $\hat Y^{-1} \dd \hat Y (y) = A (y) \dd y$, both with initial condition the
identity, Birkhoff splitting $\tilde \Phi = \tilde X^{-1} \hat Y = \hat G_- \hat G_+$ and setting 
$f = (1/2H)\sym_1 (\hat Y \hat G_+^{-1})$.
Conversely, we have the following:
\begin{proposition} \label{singdpwprop}
Let $\tilde \psi(v)$ be a singular potential which is 
non-degenerate along $J$. Integrate 
$\tilde X^{-1} \dd \tilde X  = \tilde \psi$, 
with initial condition the identity, to obtain a map,
$\tilde X : J \to \mathcal{G}$. 
Define  $\tilde \Phi: J \times J \to \mathcal{G}$ by
$$
\tilde \Phi(x,y) := \tilde X^{-1}(x) \hat Y(y), \quad \hat Y(y) := \tilde X(y).
$$
Let $\lambda \in \real \setminus \{ 0 \}$.
\begin{enumerate}
\item \label{singdpwitem1}
Set $\hat \Phi := \omega_1 \tilde \Phi$.
Then the  set $\Sigma^\circ := \hat \Phi ^{-1}(\mathcal{B}_L)$ is non-empty. 
 The map $f^\lambda: \Sigma^\circ \to \LL^3$ obtained from $\hat \Phi$  as in 
Theorem \ref{dpwthm} is a 
timelike CMC surface, regular at points where $\tilde \psi$ is regular. \\

\item \label{singdpwitem2}
Let $\Delta:= \{ (x,x) ~|~ x \in J\} \subset J \times J$. 
The set $\Sigma_s := \Sigma^\circ \cup \Delta$ is open in $J \times J$,
and the map $f^\lambda$ extends to a map $\Sigma_s \to \LL^3$ as follows: 
Set $U := \Sigma_s \cap \tilde \Phi ^{-1} (\mathcal{B}_L)$,
 which is an open set containing $\Delta$.
On $U$ perform the pointwise left normalized Birkhoff factorization
$\tilde \Phi = \hat G_- \hat G_+$. Set 
\bdm
f^\lambda := \sym_\lambda(\hat Y \hat G_+^{-1}).
\edm
The extended map $f^\lambda: \Sigma_s \to \LL^3$ is a  generalized
 timelike CMC $H$ surface.
Moreover  $\Delta$ is contained in the singular set, and is
equal to the singular set if the potential is regular.\\

\item  \label{singdpwitem3}
Along $\Delta$, we have the expressions
\beq  \label{fxfyformulae}
f^\lambda_x = \lambda \dfrac{\gamma_1}{H} \Ad_{Y^\lambda} (e_0-e_1), \quad
f^\lambda_y = - \dfrac{1}{\lambda} \dfrac{\gamma_{-3}}{H} \Ad_{Y^\lambda}  (e_0-e_1).\\
\eeq
\end{enumerate}
\end{proposition}

\begin{proof}
\emph{Item (\ref{singdpwitem1}):}
We need to show that $\hat \Phi ^{-1}(\mathcal{B}_L)$ is non-empty. The rest
of the statement then follows from Theorem \ref{firstsmallcellthm}.
Factorizing $\tilde \Phi = \tilde X^{-1} \hat Y = \hat G_- \hat G_+$ as in item (\ref{singdpwitem2}) around $\Delta$,  and writing
$$
\hat G_-= O_\infty(\lambda^{-2}) + \begin{pmatrix} 1 & b_{-1} \lambda^{-1} \\ c_{-1} \lambda^{-1}& 1 \end{pmatrix}, \quad \quad
\hat G_+ = \bbar A_0 & B_1 \lambda \\ C_1 \lambda & A_0^{-1} \ebar + O(\lambda^2),
$$
we recall from the proof of Theorem \ref{firstsmallcellthm} that 
$\hat \Phi := \omega_1 \hat G_- \hat G_+$ is in the big cell if and only if
$c_{-1} \neq 0$. Thus we need to show that $c_{-1}$ is non-zero 
away from $\Delta$,
for which it is enough to show that the derivative of $c_{-1}$ is non-zero along $\Delta$. Differentiating
 $\tilde X^{-1} \hat Y = \hat G_- \hat G_+$ and evaluating
 along $\Delta$, along which $\hat G_- = \hat G_+ = I$, 
we have
$$
- \tilde X^{-1} \dd \tilde X +  
\hat Y^{-1}\dd \hat Y
=   \dd \hat G_- +
 \dd \hat G_+.
$$
Using that  $\tilde X$ is a function of $x$ only and $\hat Y$ is a function of $y$, 
and that they take the same value along $x=y =v$, this becomes
\beqas
\dd \hat G_- + \dd \hat G_+ & =&  - A \dd x + A \dd y \\
&=&  - 2 A \dd u.
\eeqas
Comparing the coefficients of $\lambda^{-1}$, $\lambda^0$ and $\lambda$, we conclude
that, for $x=y=v$, 
\beq \label{gplusminusderivatives}
\dd \hat G_- = \bbar 0 & 2\gamma_{-3} \\ 2 \beta_1 & 0 \ebar \lambda^{-1} \dd u,
\quad \quad
\dd \hat G_+ = 2 \bbar \alpha_0 & \gamma_{-1} \lambda + \gamma_1 \lambda^3 \\
           0 & -\alpha_0 \ebar \dd u.
\eeq
 Thus, $\dd c_{-1} = 2 \beta_1 \dd u$, and  the condition that $\beta_1$ does not vanish guarantees that $\partial_u c_{-1}$
 is non-vanishing on $\Delta$.\\
 
\noindent \emph{Item (\ref{singdpwitem2}):} 
Follows from Theorem \ref{firstsmallcellthm}.\\

\noindent \emph{Item (\ref{singdpwitem3}):}
This is (\ref{limitingderivatives}) of Theorem \ref{firstsmallcellthm},
 substituting $\gamma = \gamma_1$, $\delta = -\gamma_3$,
since the potentials $\psi^X$ and $\psi^Y$ referred to there are here represented by 
$\psi^X = \hat X^{-1} \dd \hat X = \omega_1  \tilde \psi \omega_1^{-1}$ and $\psi^Y = \tilde \psi$.\\
\end{proof}

\subsection{Extending the Euclidean normal to the singular set} 
Let $f$ be a generalized timelike CMC $H$ surface.
We earlier defined the Euclidean unit normal   
$\enormal := \Ad_{e_0} \Ad_F (e_2)/ ||\Ad_F (e_2)||$, 
which is well defined on $\Sigma^\circ = \hat \Phi^{-1}(\mathcal{B}_L)$.
For a point $z_0 \in \Sigma \setminus \Sigma^\circ = \hat \Phi^{-1}(\mathcal{P}_L^1)$ 
one has,
 on some neighbourhood $U$ of $z_0$, that
 the singular set is locally given as the 
set $c_{-1} = 0$, where $c$ is the $(2,1)$-component of $\hat G_-$ in the proof of
Theorem \ref{firstsmallcellthm}. 
To extend  $\enormal$  continuously over $\hat \Phi^{-1}(\PP_L^1)$, we need to multiply it
by the sign of $c_{-1}$, and so we redefine it:
\beqa  \label{newnormal}
\enormal \,\, := \,\, \sign(c_{-1}) \frac{\Ad_{e_0} \Ad_F (e_2)}{||\Ad_F (e_2)||} 
\,\,  = \,\,  -\varepsilon  \, \sign(c_{-1}) \frac{f_x \times f_y}{\|f_x \times f_y\|},
\eeqa
where $\varepsilon = \varepsilon_1 \varepsilon_2$ as  before.
\begin{lemma}  \label{enormallemma}
Let $f$ be a generalized timelike CMC surface, locally represented by $\hat \Phi = \hat X^{-1} \hat Y$,
and let $z_0$ be a point such that $\hat \Phi(z_0) = \omega_1$. Then $\enormal$ is well defined
and smooth on a neighbourhood $U$ of $z_0$, and we have:
$$
\enormal = \frac{\Ad_{e_0} \Ad_{Y G_+^{-1}}(c_{-1} e_2 + e_0 - e_1)}{||\Ad_{Y G_+^{-1}}(c_{-1} e_2 + e_0 - e_1)||}, 
$$
where $c_{-1}:U \to \real$ and $G_+: U \to \SLC$ are smooth, $c_{-1}(z_0) = 0$ and $G_+(z_0) = I$.\\
 \end{lemma}

\begin{proof}
With notation as in the proof Theorem \ref{firstsmallcellthm}, we have
\bdm
\hat F =  \hat Y \hat H_+^{-1} = \hat Y \hat G_+^{-1} \hat U_+^{-1} D^{-1},
\edm
and
\bdm
U_+^{-1} D^{-1} = \bbar  c_{-1}^{-1} & -1 \\ 0 & c_{-1} \ebar, \quad
\Ad_{U_+^{-1} D^{-1}} (e_2) = c_{-1}^{-1}(c_{-1}e_2 + e_0 - e_1).
\edm
Substituting into the definition for $\enormal$ proves the lemma.
\end{proof}

\noindent
Note that if $Y(z_0) = I$ then this simplifies to
$\enormal(z_0) = (e_0 + e_1)/{\sqrt{2}}$.\\

\begin{lemma}  \label{enormalderlemma}
Let $f$ be a generalized timelike CMC surface, locally represented by $\hat \Phi$,
and let $z_0$ be a point such that $\hat \Phi(z_0) = \omega_1$ and $\hat Y(z_0)=I$.  Then
$$
 \lim_{z\to z_o}\dd\enormal(z)=-\frac{1}{\sqrt{2}}((\sigma+\beta)\dd u+(\sigma-\beta)\dd v)e_2.
$$ 
where,
\bdm
\psi^X= \left(O_\infty(1) + \begin{pmatrix} 0 & \beta \\ \gamma & 0 \end{pmatrix} \lambda \right) \dd x,\quad \psi^Y=\left( \begin{pmatrix} 0 & \delta \\ \sigma & 0 \end{pmatrix}  \lambda^{-1} + O(1) \right) \dd y,
\edm
are a regular potential pair corresponding to the surface.\\
\end{lemma}

\begin{proof}
As in Lemma \ref{enormallemma}, we have $\hat F =  \hat Y \hat H_+^{-1} = \hat Y \hat G_+^{-1} \hat U_+^{-1} D^{-1}$. Differentiating $\enormal$ gives:
\beqas
\dd \enormal  =&
\dfrac{\sign(c_{-1})}{\|\Ad_F(e_2)\|} \Ad_{e_0}\biggl(\Ad_F[F^{-1}\dd F,e_2] \\
    & -\dfrac{1}{\|\Ad_F(e_2)\|^2}\ip{\Ad_F[F^{-1}\dd F,e_2]}{\Ad_F(e_2)}_E \Ad_F(e_2)\biggr).
\eeqas
According to Theorem \ref{dpwthm} (3), we have 
$$
F^{-1}\dd F=\psi^X_1\dd x+\alpha_0+\hat H_+ \big|_{\lambda=0} \psi^Y_{-1}\hat H_+^{-1}\big|_{\lambda=0}, 
$$
where $\alpha_0$ is a diagonal matrix of $1$-forms. As in the proof of
Theorem \ref{firstsmallcellthm}, we write 
$$
\psi^X_1=\bbar 0 & \beta \\ \gamma & 0 \ebar,\qquad \hat H_+ \big|_{\lambda=0} \psi^Y_{-1}\hat H_+^{-1}\big|_{\lambda=0} =\bbar 0 & g^2\delta \\ g^{-2}\sigma & 0 \ebar,
$$
where $g = c_{-1} A_0 \to 0$ as $z_0$.
Writing $c=c_{-1}$ to simplify notation, one
obtains from  $\Ad_{YG_+^{-1}}=I+o(1)$, the following formula:  
\beqas
\Ad_F[F^{-1}\dd F,e_2]=& 2\bbar c\gamma & \gamma+c^{-2}\beta \\ -c^2\gamma & -c\gamma \ebar \dd x \\
& +  2\bbar cg^{-2}\sigma  g^{-2}\sigma+c^{-2} g^2\delta \\ -s_{-1}^2 g^{-2}\sigma & -cg^{-2}\sigma \ebar\dd y
 +o(1).
\eeqas
Using that $\Ad_F(e_2)=c^{-1}(e_0-e_1)+e_2+o(1)$, we obtain
\begin{equation*}
\begin{split}
\ip{\Ad_F[F^{-1}\dd F,e_2]}{\Ad_F(e_2)}_E =      &-2(c\gamma+c^{-1}(c^{-2}\beta+\gamma))\dd x\\
&-2(c g^{-2}\sigma+c^{-1}(c^{-2} g^2\delta+g^{-2}\sigma))\dd y+o(1),
\end{split}
\end{equation*}
and so we have 
\begin{equation*}
\begin{split}
\Ad_{e_0}(\dd\enormal)=
&o(1) + \frac{1}{\sqrt{2}}\bbar c^2\gamma-c^4\gamma-\beta & 2c^2\gamma-2c^3\gamma-2c\gamma \\ -2c^2\gamma & -c^2\gamma+c^4\gamma+\beta\ebar\dd x\\
&+\dfrac{c^2}{g^2\sqrt{2}}\bbar \dfrac{c^2}{g^2} \sigma-\dfrac{c^4}{g^2}\sigma+g^2 \delta & -2g^{-2}\sigma \\ -2 g^{-2}\sigma & -\dfrac{c^2}{g^2} \sigma+\dfrac{c^4}{g^2}\sigma+g^2\delta \ebar \dd y.
\end{split}
\end{equation*}
As $z\to z_0$ we have $c\to 0$ and $g/c \to 1$; hence 
$$
\Ad_{e_0}(\dd\enormal)\to\frac{1}{\sqrt{2}}(\beta e_2\dd x -\sigma e_2\dd y)=\frac{1}{\sqrt{2}}((\beta+\sigma)\dd u+(\beta-\sigma)\dd v)e_2.
$$
Since $\Ad_{e_0}(e_2)=-e_2$, the result follows. \\
\end{proof}

\begin{proposition}  \label{nondegsingprop}
Let $f$ be the surface constructed from a singular potential $\tilde \psi$ in 
accordance with Proposition \ref{singdpwprop}.  
The map $f: \Sigma_s \to \LL^3$ is a frontal.
A singular point on $\hat \Phi^{-1}(\mathcal{P}_L^1)$ is non-degenerate if and only if
$\tilde \psi$ is non-degenerate and regular at the point.\\
\end{proposition}
\begin{proof}
That $f$ is a frontal follows from Lemma \ref{enormallemma}. To show that the frontal is non-degenerate, we must show  that $\dd \chi \neq 0$
 at a singular point $z_0$, where $f_x \times f_y = \chi \enormal$.
By the definition \eqref{newnormal}, we have 
\bdm
\chi = -\varepsilon \sign(c_{-1}) \|f_x \times f_y\|.
\edm
In the notation of Theorem \ref{firstsmallcellthm} we have 
$$
\varepsilon\frac{e^\omega}{2}=\ip{f_x}{f_y}_L=-2\frac{\gamma\delta A_0^2 c_{-1}^2}{H^2}.
$$

Substituting into the expression \eqref{fxcrossfy} we obtain
$$
\chi = -\varepsilon \sign(c_{-1}) \frac{e^\omega}{2}\|\Ad_F(e_2)\|=
    2  \sign(c_{-1}) \frac{\gamma\delta A_0^2c_{-1}^2}{H^2}\|\Ad_F(e_2)\|.
$$
The derivative is
\beqas
\dd \chi = &
 \sign(c_{-1}) \dfrac{\gamma\delta A_0^2}{H^2}
\biggl(
 2 c_{-1}^2\dfrac{\ip{\Ad_F[F^{-1}\dd F,e_2]}{\Ad_F(e_2)}_E}{\|\Ad_F(e_2)\|}  \\
&
+  4 c_{-1}\|\Ad_F(e_2)\|\dd c_{-1}
\biggr) +o(1).
\eeqas
From the proof of Lemma \ref{enormalderlemma} and the fact that $c_{-1}\|\Ad_F(e_2)\|=\sign(c_{-1})\sqrt{2}+o(1)$, an easy calculation gives
$$
c_{-1}^2\frac{\ip{\Ad_F[F^{-1}\dd F,e_2]}{\Ad_F(e_2)}_E}{\|\Ad_F(e_2)\|}=-\sqrt{2}\sign(c_{-1})(\beta\dd x+\sigma\dd y)+o(1).
$$
With our choice of potentials, we have 
\bdm
\beta=\beta_1, \quad \gamma= \gamma_1 \quad \sigma = -\beta_1,  \quad \delta = -\gamma_{-3}, 
\edm
so that ($\beta \dd x + \sigma \dd y) = 2\beta_1 \dd u$. From
\eqref{gplusminusderivatives} we have
  $\dd c_{-1}=2\beta_1\dd u$. Hence
$$
\dd \chi = -4\sqrt{2} \dfrac{\beta_1 \gamma_1 \gamma_{-3}}{H^2}\dd u.
$$
Hence  the singular point is non-degenerate if and only if $\gamma_1$, $\gamma_{-3}$ and
$\beta_1$ are non-zero, which is the condition that the potential is regular and non-degenerate.\\

\end{proof}

\subsection{The singular geometric Cauchy problem}   \label{singgcpsect}
The goal of this section is to construct  generalized timelike CMC surfaces with 
prescribed singular curves. As above, we assume the curve in the coordinate domain
is non-characteristic, that is, never parallel to a coordinate line. \\

\noindent
In order to obtain
a unique solution, we need to specify the derivatives of $f$ as well, as follows:\\
\begin{problem}  \label{ncsgcpprob}
The (non-characteristic) singular geometric Cauchy problem:
Let $J$ be a real interval with coordinate $v$. Given a smooth map $f_0: J \to \LL^3$, and a vector field
$V: J \to \LL^3$ such that $f_0^\prime(v)$ is lightlike,  $V$ is proportional
to $f_0^\prime(v)$ and the two vector fields do not vanish simultaneously.
Find a generalized timelike CMC surface $f: \Sigma \to \LL^3$, where $\Sigma$ is some open
subset of the $uv$-plane which contains the interval $J \subset \{u = 0 \}$, that, away
from $J$, is conformally immersed with lightlike coordinates $x=u+v$, $y=-u+v$, and
 such that along $J$
the following hold:
\bdm
\left. f \right |_J = f_0, \quad \left. f_u \right|_J = V.
\edm
\end{problem}

\noindent
After an isometry of the ambient space, we can assume that
$f_0^\prime(v) = s(v)(-e_0 + \cos \theta(v) e_1 + \sin \theta(v) e_2)$,
for some smooth functions $s$ and $\theta$, with $\theta(0)=0$, and so the
derivatives of a solution $f$ must satisfy:
\beqa
\begin{split}  \label{fufveqns}
f_v = s(-e_0 + \cos \theta e_1 + \sin \theta e_2),\\
f_u= t(-e_0 + \cos \theta e_1 + \sin \theta e_2),
\end{split}
\eeqa
where  $s$, $t$ are smooth and do not vanish simultaneously, and the function
$t$ is deduced from $V$.\\

\noindent
We want to construct  a singular potential 
\bdm
\tilde \psi =\bbar -\alpha_0  & -\gamma_1 \lambda^3 - \gamma_{-1} \lambda -\gamma_{-3} \lambda^{-1}  \\ -\beta_1 \lambda^{-1}  & \alpha_0 \ebar \dd v.
\edm
for the surface. Our task is to find $\gamma_1$, $\gamma_{\pm 1}$, $\gamma_{-3}$ and $\alpha_0$. 
We begin by looking for a 
"singular frame" $F_0 = \tilde X\big|_{\lambda =1} = \hat Y\big|_{\lambda =1}$ 
along $y=x$, such that $F_0^{-1} \dd F_0 = \tilde \psi \big|_{\lambda=1}$.
According  to Proposition \ref{singdpwprop}, 
using the formulae (\ref{fxfyformulae}) for $f_y$ and $f_x$ along $\Delta$, 
we must have
\beqas
f_v &=& \frac{1}{H}(-\gamma_1+\gamma_{-3}) \Ad_{F_0}(e_1-e_0),  \\
f_u &=&-\frac{1}{H}(\gamma_1+\gamma_{-3}) \Ad_{F_0}(e_1-e_0).
\eeqas
Comparing with  (\ref{fufveqns}) a
solution for $F_0$ is  given by:
\bdm
F_0 = \bbar  \cos( \theta/2) & -\sin(\theta/2) \\
  \sin(\theta/2) &  \cos (\theta/2) \ebar,
\edm
and with that choice of $F_0$, the functions $\gamma_1$ and $\gamma_{-3}$ are determined as:
\bdm
\gamma_1 = \frac{-H(s+t)}{2}, \quad
\gamma_{-3} = \frac{H(s-t)}{2}.\\
\edm

\noindent
Next we have the expression
\bdm
F_0^{-1} (F_0)_v = \frac{\theta^\prime}{2} \bbar 0 & -1 \\ 1 & 0 \ebar.
\edm
Comparing this with $\tilde \psi$, evaluated at $\lambda =1$,
we obtain:
$\alpha_0 =0$, 
$\theta^\prime/2 = -\beta_1$, and
$\theta^\prime/2 = \gamma_1+\gamma_{-1}+\gamma_{-3}$, and so:
\bdm
\alpha_0=0, \quad \beta_1 = -\frac{\theta^\prime}{2}, \quad
\gamma_{-1} = \frac{\theta^\prime}{2} + Ht.\\
\edm

\noindent
Hence, provided that $\hat \Phi^{-1} (\mathcal{B}_L)$ is not empty,
 a solution for the 
singular geometric Cauchy problem with data given by (\ref{fufveqns}) is obtained 
from the singular potential,
\bdm
\tilde \psi = \frac{1}{2}\bbar 0  & H(s+t) \lambda^3 - \left(\theta^\prime
 + 2Ht \right) \lambda + H(t-s) \lambda^{-1}  \\ \theta^\prime \lambda^{-1}  & 0 \ebar \dd v.
\edm
According to Proposition \ref{nondegsingprop}, the singular curve
 is non-degenerate if and only if the three functions $s+t$, $t-s$ and $\theta^\prime$
 do not vanish.  The non-degeneracy condition is thus:
\bdm
s \neq \pm t \quad \textup{and} \quad \theta^\prime \neq 0.\\
\edm

\begin{theorem} \label{singgcptheorem}
The surface $f: J \times J \to \LL^3$ obtained from the singular potential $\tilde \psi$
given above is the \emph{unique} solution for the non-characteristic
geometric Cauchy problem given by the equations (\ref{fufveqns}).
\end{theorem}
\begin{proof}
We know that any solution surface is given locally by the the construction in 
Proposition \ref{singdpwprop}.  So suppose we have another solution $\check f$, 
with corresponding singular potential $\check {\tilde \psi}$.
From the formulae (\ref{fxfyformulae}) for $f^\lambda_x$ and $f^\lambda_y$,
we must have, along $\Delta$,
\bdm
\gamma_1 \Ad_{Y^\lambda}(e_1-e_0) = \check \gamma_1 \Ad_{\check Y^\lambda}(e_1-e_0).
\edm
We conclude that
\bdm
\check Y^\lambda(y) = Y^\lambda(y) T^\lambda(y),
\edm
where $T^\lambda$ commutes, up to a scalar,
 with $(e_1-e_0)$, and is therefore of the form
\bdm
T^\lambda = \bbar  \mu &  \nu \\ 0 &  \mu^{-1} \ebar.
\edm
Now computing $(\check Y^\lambda)^{-1} \dd \check Y^\lambda =
  \Ad _{T^\lambda} \tilde \psi + (T^\lambda)^{-1} \dd T^\lambda$,
we obtain for the $(1,1)$ and $(2,1)$ components respectively:
\beqas
-\check \alpha_0 = -\alpha_0 +  \mu  \nu \beta_1 \lambda^{-1} + \mu ^{-1} \dd  \mu,\\
- \check \beta_1 \lambda^{-1} = - \mu^2 \beta_1 \lambda^{-1}.
\eeqas
It follows that 
\bdm
\mu = \mu_0, \quad \nu = \nu_1 \lambda,
\edm
where $\mu_0$ and $\nu_1$ are constant in $\lambda$.\\

\noindent
Now the surface $f$ is obtained as 
$2H f = \sym _1 \left(\hat Y \, \hat G_+^{-1}\right)$,
where 
\beq   \label{firstbirkhoff}
\hat X^{-1} \, \hat Y = \hat G_- \, \hat G_+
\eeq
 is a normalized Birkhoff factorization.
Likewise, since $\hat{\check Y} = \hat Y \, \hat T$, and
$\hat{\check X}(x) = \hat{\check Y}(x) = \hat X(x) \, \hat S(x)$,
where we set $\hat S(x) = \hat T(x)$,  the map
 $\check f$ is obtained as 
$2 H \check f = \sym _1 (\hat Y \, \hat T \, \hat {\check G}_+^{-1})$, where
\bdm
 \hat S^{-1}  \, \hat X^{-1} \, \hat Y \,\hat T =
      \hat {\check G}_- \, \hat {\check G}_+.
\edm
Now, inserting the Birkhoff factorization at (\ref{firstbirkhoff}), we have
\beqas
\hat {\check G}_- \, \hat {\check G}_+ & =& 
  \hat S^{-1} \, \hat G_- \, \hat G_+   \,\hat T \\
  &=& \hat H_- \, \hat D_+  \,  \hat G_+   \,\hat T,
\eeqas
where $\hat H_-$ takes values in $\mathcal{G}_*^-$ and, writing 
$\hat G_- = \small{\bbar A & B \\ C & D \ebar}$, we have, if $v \neq 0$,
\bdm
\hat D_+ = \bbar r & \lambda \\ 0 & r^{-1} \ebar, \quad 
r = \frac{\mu_0 A_0 + \nu_1 C_{-1}}{\nu_0 D_0}, 
\edm
and  $\hat D_+ = \hat S$ if $\nu = 0$.
Since $\hat T$ takes values in $\mathcal{G}^+$, we conclude by uniqueness of
the normalized Birkhoff factorization of 
$\hat {\check G}_- \, \hat {\check G}_+ $ that
\bdm
\hat {\check G}_+  = \hat D_+  \,  \hat G_+   \,\hat T,
\edm
and 
\beqas
2H \check f(x,y) &=&\sym _1 \left(\hat Y \, \hat T \,
  \, \hat T^{-1} \hat G_+^{-1} \hat D_+^{-1} \right) \\
  &=& \sym _1 \left(\hat Y \, \hat G_+^{-1} \hat D_+^{-1} \right) \\
  &=& \sym _1 \left(\hat Y \, \hat G_+^{-1} \right),
\eeqas
because  right multiplication by either of the candidates
for $\hat D_+^{-1}$ leaves the Sym formula unchanged.
Thus $\check f = f$, and the solution is  unique.\\
\end{proof}


\subsection{Generic singularities}   \label{gensingsect}
The object of this section is to  prove:
\begin{theorem} \label{singtheorem}
Let $f$ be a generalized timelike CMC surface, and $z_0$ a
 non-degenerate singular point.
Assume that the singular curve is non-characteristic at $z_0$.  
By Theorem \ref{singgcptheorem}, we
 may assume that $f$ is locally represented by a singular potential:
 \bdm
\tilde \psi = \frac{1}{2}\bbar 0  & H(s+t) \lambda^3 - \left(\theta^\prime
 + 2Ht \right) \lambda + H(t-s) \lambda^{-1}  \\ \theta^\prime \lambda^{-1}  & 0 \ebar \dd v,
\edm
where $s$, $t$ and $\theta^\prime$ are the geometric Cauchy data described in that section,
$z_0=(0,0)$,
and $\tilde X(0) = \hat Y(0) = I$.\\

\noindent
Then, at $z_0 = (0,0)$, the surface is locally diffeomorphic to a :
\begin{enumerate}
\item \label{singthm1}  cuspidal edge if and only if
both $s(0)$ and $t(0)$ are nonzero,\\
\item   \label{singthm2}
swallowtail if and only if 
\bdm
s(0) = 0, \quad s^\prime(0) \neq 0, \quad t(0) \neq 0,\\
\edm
\item \label{singthm3}
cuspidal cross cap if and only if
\bdm
t(0) = 0, \quad t^\prime(0) \neq 0, \quad s(0) \neq 0.\\
\edm
\end{enumerate}
\end{theorem}

Before proving this, we state conditions suitable for our context that
characterize swallowtails,  cuspidal edges and cuspidal cross caps: 
\begin{proposition} \label{swallowtailprop}
\cite{krsuy}.
Let $f: U \to \R^3$ be a front, and $p$ a non-degenerate singular point.
Suppose that $\gamma: (-\delta, \delta) \to U$ is a local parameterization 
of the singular curve, with parameter $x$ and tangent vector $\dot\gamma$, 
and  $\gamma(0)= p$.
 Then:
\begin{enumerate}
\item The image of $f$ in a neighbourhood of $p$ is diffeomorphic to a
cuspidal edge if and only if $\eta(0)$ is not proportional to $\dot\gamma(0)$.
\item The image of $f$ in a neighbourhood of $p$ is diffeomorphic to a
swallowtail if and only if $\eta(0)$ is proportional to $\dot\gamma(0)$ and
\bdm
\frac{\dd}{\dd x} \det \left(\dot\gamma(x), \eta(x)\right)\Big|_{x=0} \neq 0.
\edm\\
\end{enumerate}
\end{proposition}

\begin{theorem} \cite{fsuy}. \label{fsuythm}
Let $f: U \to \real^3$ be a  frontal, with Legendrian lift $L = (f, \enormal)$,
 and let $z_0$ be a non-degenerate singular point.
Let $Z: V \to \real^3$ be an arbitrary differentiable function on a neighbourhood $V$ of $z_0$
such that:
\begin{enumerate}
\item  \label{transcond1}
$Z$ is orthogonal  to  $\enormal$.
\item \label{transcond2}
$Z(z_0)$ is transverse to the subspace $\dd f(T_{z_0}(V))$.\\
\end{enumerate}

\noindent
Let  $x$ be the parameter for the singular curve, 
$\eta(x)$ a choice vector field for the null direction,
 and set
\bdm
\tau (x) := \ip{\enormal}{ \dd Z  (\eta)}_E \big|_x.
\edm
The frontal $f$ has a \emph{cuspidal
cross cap} singularity at  $z=z_0$ if and only:
\begin{itemize}
\item [(A)] \label{conda} $\eta(z_0)$ is transverse to the singular curve;
\item [(B)] \label{condb} $\tau(z_0) = 0$ and $\tau^\prime(z_0) \neq 0$.\\
\end{itemize}
\end{theorem}

\begin{proof}[Proof of Theorem \ref{singtheorem}] 
First, note that $f$ is a front if and only if $t$ does not vanish, since, from
Lemma \ref{enormalderlemma} we have
 $\dd \enormal\big|_{z=z_0} = -\sqrt{2}\theta^\prime/2 \, \dd v \, e_2$,
and, from the geometric Cauchy construction, 
$\dd f\big|_{z=z_0} = (s \dd v + t \, \dd u)(e_1 -e_0)$. 
 Thus, writing $L=(f,\enormal)$, we have
\bdm
L_u = \left(t(e_1 -e_0), \, 0 \right), \quad  L_v = \left( s(e_1 -e_0), \, -\frac{\sqrt{2} \theta^\prime}{2} e_2 \right).
\edm  
The curve is assumed non-degenerate, so $\theta^\prime \neq 0$, and therefore $L$ has rank $2$ at
$z_0$ if
and only if $t(0) \neq 0$.\\

\noindent
The singular curve is given by $u=0$ and hence tangent to $\partial_v=\partial_x+\partial_y$, and the null direction is defined by the vector field $\eta= s \partial_u - t\partial_v$. Hence, by Proposition \ref{swallowtailprop} the surface is locally diffeomorphic to a cuspidal edge around the singular point $z_0$ if and only if both $s$ and $t$ are non-zero. This proves item (\ref{singthm1}). To prove item (\ref{singthm2}), we just need to notice that $\det(\dot\gamma,\eta)=-s$. \\

\noindent
To prove item (\ref{singthm3}), we will choose a suitable vector field $Z$ and apply
Theorem \ref{fsuythm} above.
We use the setup from Lemma \ref{enormallemma}, whence we see that the Euclidean normal around the singular point is parallel to 
$$
\Ad_{e_0}\Ad_{YG_+^{-1}}(e_0-e_1+c_{-1}e_2).
$$
Furthermore, $f_x$ and $f_y$ are both parallel to $\Ad_{F_0}(e_0-e_1)=\Ad_{e_0}\Ad_{F_0}(e_0+e_1)$ along the singular curve. Thus, the vector field $Z$ defined by 
\begin{equation*}
Z= \Ad_{e_0}\Ad_{YG_+^{-1}}(e_0-e_1+c_{-1}e_2)\times \Ad_{e_0}\Ad_{YG_+^{-1}}(e_0+e_1),
\end{equation*}
is orthogonal to the Euclidean normal in a neighbourhood of $z_0$ and transverse to $f_x$ and $f_y$ along the singular curve in this neighbourhood.  From Section \ref{enormalsect}, we
have $e_0 \times e_1 = e_2$, $e_1 \times e_2 = e_0$ and $e_2 \times e_0 = e_1$, and for any vectors $a$ and $b$ and matrix $X$ we have 
$(\Ad_{e_0}\Ad_X(a)) \, \times \, (\Ad_{e_0}\Ad_X(b)) = \Ad_X \Ad_{e_0}(a \times b)$.
 Thus,
\beqas
Z =&  \Ad_{YG_+^{-1}}\Ad_{e_0}(c_{-1}(e_1-e_0)+2e_2)\\
 =&  \Ad_{YG_+^{-1}}(-2e_2 - c_{-1}(e_1+e_0)).
\eeqas
 Write $\tilde \psi = \hat A(v) \dd v$, so that
$\hat Y ^{-1} \dd \hat Y = \hat A(y) \dd y$.  Along the singular curve, where $c_{-1}=0$,
 we have$$
\dd Z = -\Ad_{F_0}([A\,  \dd y-\dd G_+,2e_2] + (e_0+e_1)\dd c_{-1}).
$$
From (\ref{gplusminusderivatives}) with $\lambda =1$, we have 
\bdm
\dd  G_+ = \frac{1}{2}(\theta^\prime + H(t-s))(e_1-e_0) \dd u,  \quad
\quad \dd c_{-1} = - \theta^\prime \dd u,
\edm
and we also have $A = \hat A\big|_{\lambda=1} = \theta^\prime e_0 /2$.\\

\noindent
Hence 
\beqas
(A \dd y-\dd G_+)(\eta) &=&  \frac{1}{2} (\theta^\prime e_0 (\dd v - \dd u) - 
   (\theta^\prime + H(t-s))(e_1-e_0) \dd u) (s \partial_u - t \partial v)),\\
  &=& \frac{1}{2}\left( (sH(t-2) - t \theta^\prime)e_0 
    - s(\theta^\prime + H(t-s))e_1 \right),
\eeqas
and
\bdm
\dd c_{-1}(\eta) = -s \theta^\prime.\\
\edm

\noindent
 Putting all these together:
\bdm
\dd Z(\eta) \big|_{u=0} = -  \Ad_{F_0}\left(s\left( 2 H (t-s) + \theta^\prime \right)(e_0-e_1)
   + 2 t \theta^\prime e_1 \right).
\edm
Along the singular curve the expression for $\enormal$ simplifies to
\bdm
\enormal = \frac{1}{\sqrt{2}}\Ad_{F_0}(e_0 + e_1).
\edm
Since $F_0$ is in $\textup{SU}(2)$ and preserves the Euclidean inner product, 
we finally arrive at 
$$
\tau(v) = \ip{\enormal}{\dd Z(\eta)}_E(0,v)= -\sqrt{2} t \theta^\prime.
$$
Since $\theta^\prime \neq 0$, the condition (B) of Theorem \ref{fsuythm} is
equivalent to: $t=0$ and $t^\prime \neq 0$; finally, condition (A) is
equivalent to: $s \neq 0$. This proves item (\ref{singthm3}). \\
\end{proof}

\section{Prescribing class II singularities of characteristic type}   \label{chartypesect}
\noindent
Suppose now that we have a generalized timelike CMC surface with non-degenerate singular curve that is always tangent to a characteristic direction, that is, the curve is given in local lightlike coordinates $(x,y)$ as $y=0$. \\

\noindent
If $\hat X$ and $\hat Y$ are the
associated data, and the singularity is of class II, then we must have 
$\hat \Phi(x,0) = \hat X^{-1}(x) \, \hat Y(0) = \hat H_-(x) \omega_1 \hat H_+(x)$, 
where $\hat H_{\pm}$ take values in $\mathcal{G}^\pm$.  
By a similar argument to that in Section \ref{singpotsection}, no generality is
lost in assuming that $\hat H_-(x) = I$, and 
\bdm
\hat X(x) = \hat H_+(x) \omega_1^{-1}, \quad \hat H_+(x) \in \mathcal{G}^+, \quad
\hat H_+(0) = I, \quad
\hat Y(0) = I.
\edm
Now writing
\bdm
\hat X^{-1} \dd \hat X = (... + A_{-1} \lambda^{-1} + A_0 + A_1 \lambda )\dd x,
\quad
\hat H_+ ^{-1} \dd \hat H_+ = (B_0 + B_1 \lambda + ....) \dd x,
\edm
and computing $\hat X^{-1} \dd \hat X = \omega_1(\hat H_+ ^{-1} \dd \hat H_+) \omega_1^{-1}$, we conclude, comparing coefficients of like powers of $\lambda$, that 
\bdm
\hat X^{-1} \dd \hat X = \bbar \alpha_0 & 0 \\ \gamma_{-1} \lambda^{-1} + \gamma_1 \lambda & -\alpha_0 \ebar \dd x,
\edm
where $\alpha_0$, $\gamma_{-1}$ and $\gamma_1$ are independent of $\lambda$ and all other coefficients are zero.
The "singular frame" $\tilde X = \hat X \omega_1 = \hat H_+$ then has Maurer-Cartan form
\bdm
 \tilde X^{-1} \dd \tilde X = \bbar -\alpha_0 & -\gamma_{-1} \lambda - \gamma_1 \lambda^3 \\ 0 & \alpha_0 \ebar \dd x.
\edm
\begin{definition}
Let $J_x$ and $J_y$ be a pair of open intervals each containing $0$.
A \emph{characteristic singular potential pair} $(\tilde \psi^X,  \psi^Y)$
is a pair of $\textup{Lie}(\mathcal{G})$-valued 1-forms on $J_x \times J_y$,
 the Fourier expansions in $\lambda$ of which are of the form
\bdm
\tilde \psi^X =  \bbar -\alpha_0 & -\gamma_{-1} \lambda - \gamma_1 \lambda^3 \\ 0 & \alpha_0 \ebar \dd x,
\quad
\psi^Y = \left(\bbar 0 & \delta \lambda^{-1} \\ \sigma \lambda^{-1} & 0 \ebar  + O(1) \right) \dd y.
\edm
The potential is semi-regular if $\gamma_1$ and $\delta$ do not vanish simultaneously,
and regular at points where both are non-zero.
\end{definition}\
By Theorem \ref{firstsmallcellthm}, integrating 
 $\tilde X^{-1} \dd \tilde X = \tilde \psi^X$, and
 $\hat Y^{-1} \dd \hat Y = \psi^Y$, both with initial condition the identity,
 a generalized timelike CMC surface is produced, provided $\hat \Phi = \omega_1 \tilde X^{-1} \hat Y$ maps some open set into the big cell. 
Since 
$\omega_1^{-1} \hat \Phi(x,0) = \tilde X (x)^{-1} = \hat H_+(x)^{-1}$,
the Birkhoff decomposition $\omega_1^{-1} \hat \Phi  = \hat G_- \hat G_+$ 
used in Theorem \ref{firstsmallcellthm} reduces to 
\bdm
\hat G_+(x,0) = \tilde X(x)^{-1}, \quad \hat G_-(x,0) = I,
\edm
and the surface along $y=0$ is given by
\bdm
f^\lambda(x,0) = \sym _\lambda(\tilde X(x)).
\edm
The limiting derivatives of $f^\lambda$ along $y=0$ are given, by (\ref{limitingderivatives}), as
\bdm
f_x^\lambda = \lambda\frac{\gamma_1}{H} \Ad_{\tilde X}(e_0-e_1), \quad
f_y^\lambda = \lambda ^{-1} \frac{\delta}{H} \Ad_{\tilde X}(e_0-e_1).\\ 
\edm

\noindent
As in the non-characteristic case, the general geometric Cauchy problem
is to find a solution $f$, this time with $f(x,0)=f_0(x)$ prescribed and
 which, along $y=0$, satisfies:
 \beqas
f_x = s(-e_0 + \cos \theta e_1 + \sin \theta e_2),\\
f_y= t(-e_0 + \cos \theta e_1 + \sin \theta e_2),
\eeqas
with $\theta(0)=0$. 
Comparing with the above equations for $\tilde X$, a solution $F_0$ for 
$\tilde X \big|_{\lambda=1}$, together with the functions $\gamma_1$ and $\delta$, is:
\bdm
F_0 = \bbar  \cos( \theta/2) & -\sin(\theta/2) \\
  \sin(\theta/2) &  \cos (\theta/2) \ebar, \quad
  \gamma_1 = -sH, \quad \delta = -tH.
\edm
Since $\delta$ is a function of $y$ only, we must have 
\bdm
t(x) = t_0 = \textup{constant},
\edm
which is one way to see that these singularities are not generic.\\

\noindent
Computing $F_0^{-1} \dd F_0$ and equating it with $\tilde \psi^X \big|_{\lambda=1}$,
we conclude that $\theta^\prime = 0$, so that the curve is a straight line, with:
\bdm
\theta=0, \quad \alpha_0=0,  \quad \gamma_{-1} =  s H.
\edm
Thus the general characteristic geometry Cauchy problem is in fact:
\beqas  
f_x = s(-e_0 +  e_1),\\
f_y= t_0(-e_0 +  e_1 ),
\eeqas
with a solution  given by the characteristic singular potential pair:
\beqas
\tilde \psi^X =  \bbar 0 & - s H \lambda + s H \lambda^3 \\ 0 & 0 \ebar \dd x,
\\
\psi^Y = \left(\bbar 0 & \delta \lambda^{-1} \\ \sigma \lambda^{-1} & 0 \ebar   + O(\lambda^2) \right)\dd y, \quad \delta(0)= - t_0 H, 
\eeqas
where  $\sigma$ is an arbitrary function of $y$, as are the higher
 order terms of $\psi^Y$.\\

\noindent
As in the proof of Theorem \ref{singgcptheorem}, one can show that
any other solution $\check X$ for $\hat X$ must be of the form
$\check X = \hat X T$ where $T$ is a diagonal matrix constant in $\lambda$ and has no effect on the solution surface.  Hence the potential pair $\left(\tilde \psi^X,\, \psi^Y \right)$ above   represents the most general solution for the 
characteristic singular geometric Cauchy problem of class II.\\

\begin{figure}[here]  
\begin{center}
\includegraphics[height=36mm]{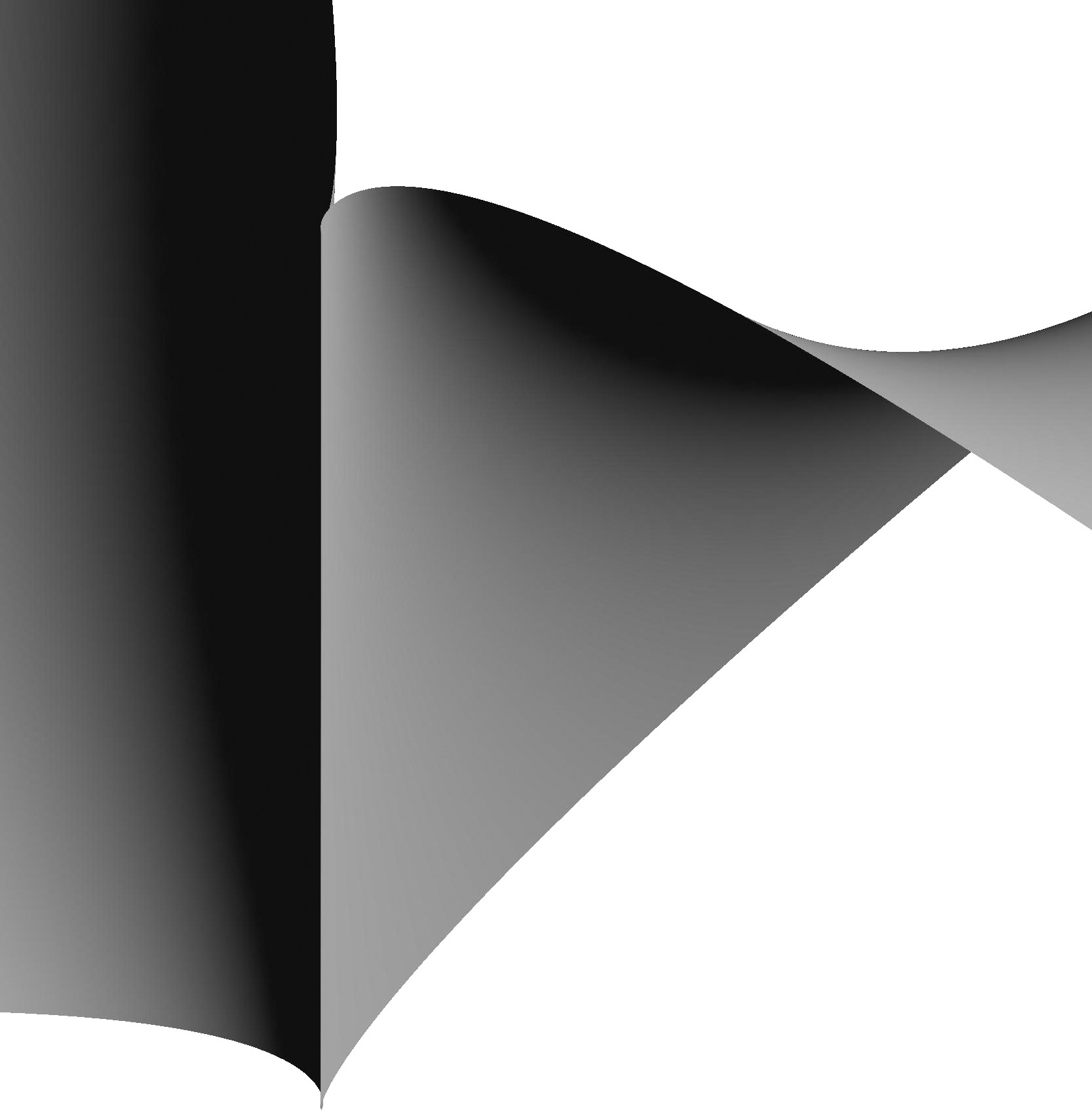} \hspace{1cm}
\includegraphics[height=36mm]{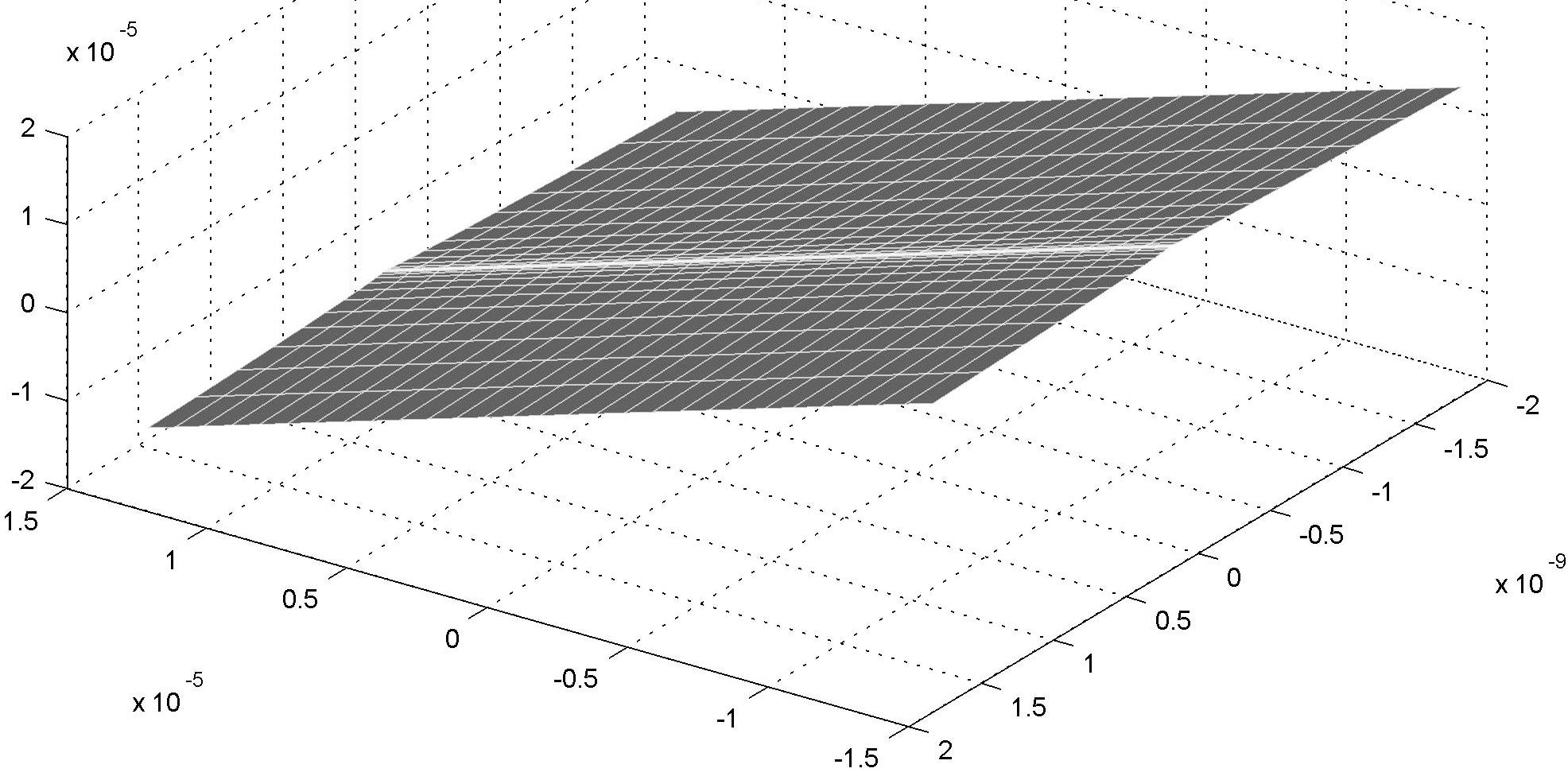} 
\end{center}
\caption{Numerical plots of solutions to the characteristic geometric Cauchy problem.
Left: $s(x)= 1$, $\delta=\sigma=1$. Right: $s(x)=1$, $\delta(y)=y$, $\sigma(y)=1$.}
\label{figure3}
\end{figure}

\noindent
Finally, to determine the  condition that ensures that the values of the map $\hat \Phi$
are not constrained to the small cell: as in Theorem \ref{firstsmallcellthm}, the surface is obtained from $\tilde \Phi = \tilde X^{-1} \hat Y = \hat G_- \hat G_+$, and $\hat \Phi = \omega_1 \tilde \Phi$ maps some point
into the big cell provided that, at some point, $\dd c_{-1} \neq 0$, where 
\bdm
\hat G_- = O_\infty(\lambda^{-2}) + \bbar 1 & b_{-1} \lambda^{-1} \\ c_{-1} \lambda^{-1} & 1 \ebar,
\edm
Evaluating
derivatives at $(0,0)$, we find that $\dd c_{-1}(0,0) = \sigma(0,0)$, and so
the non-degeneracy condition for the potential is
\bdm
\sigma(0) \neq 0.\\
\edm

\noindent
We do not analyze the types of singularities involved here, but two examples of solutions are illustrated in Figure \ref{figure3}, one appearing to be a cuspidal edge and the other appearing to be a singularity of the parameterization, rather than a true geometric singularity.\\

\section{Examples of degenerate singularities}  \label{numsect}

\noindent Examples of the way various degenerate geometric Cauchy data impact the resulting construction are illustrated in Figures \ref{figure4} 
and \ref{figure5}.\\


\begin{figure}[ht]
\centering
$
\begin{array}{cc}
\includegraphics[height=25mm]{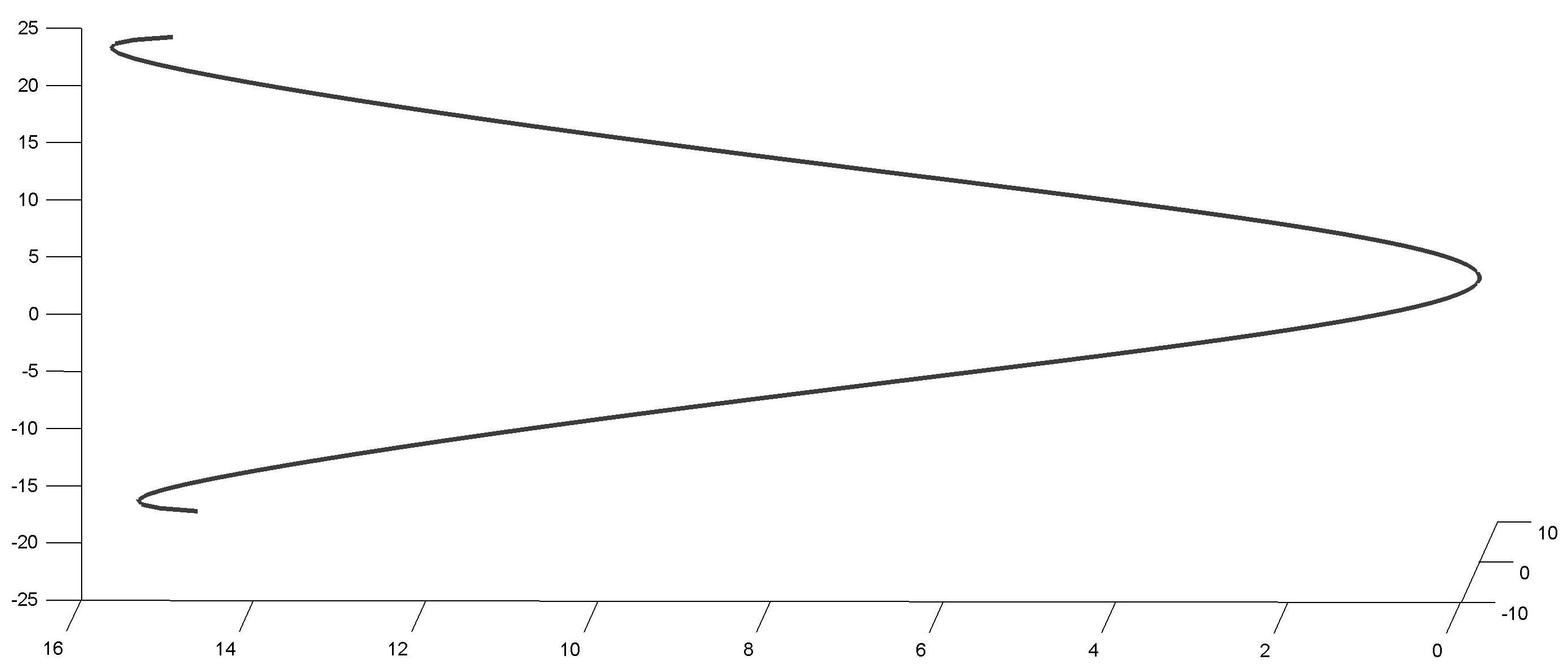} \quad & \quad
\includegraphics[height=25mm]{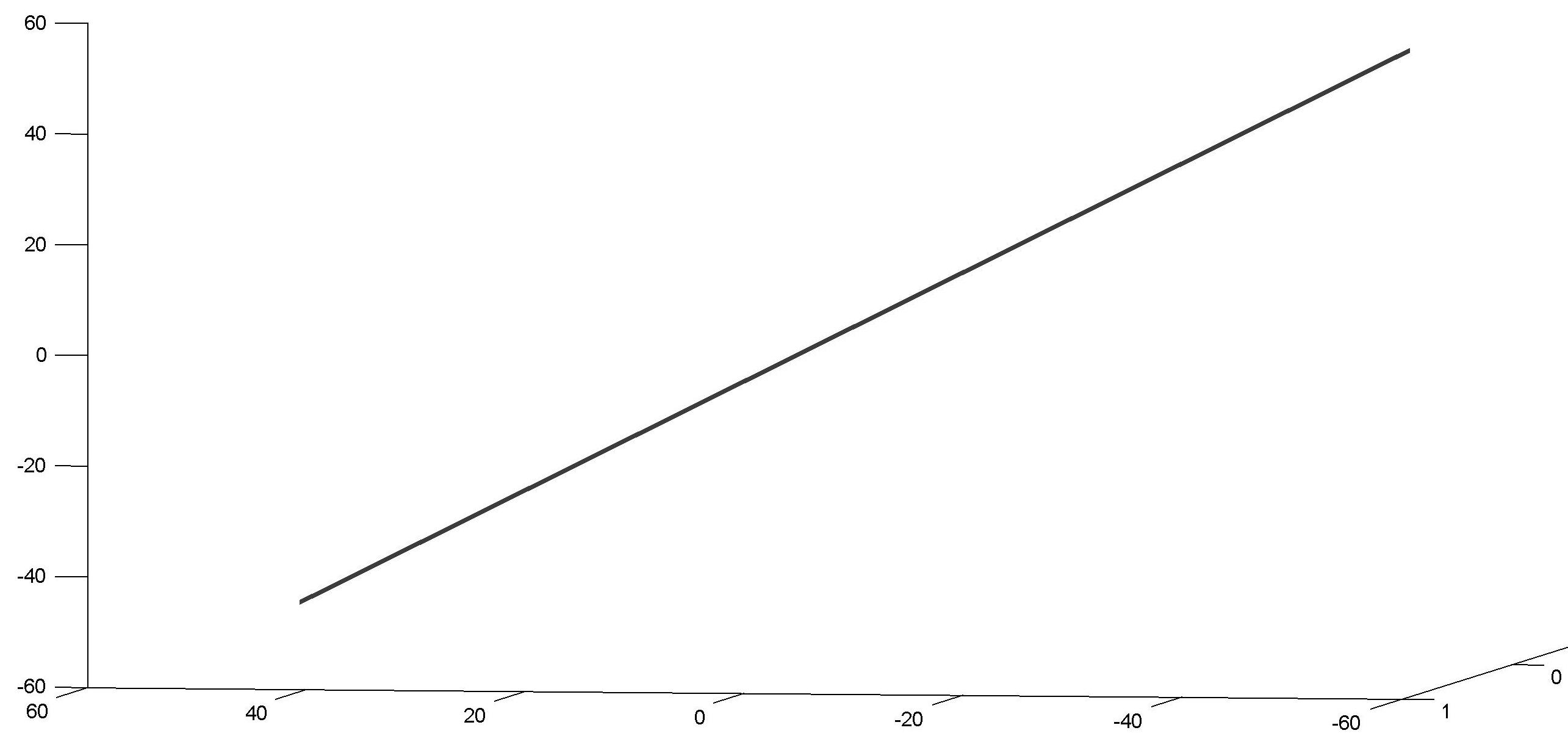} \\
s(v)=t(v)=1, \,\, \theta^\prime(v) = 0.1.
\quad & \quad  s(v)=3, \,\, t(v)=2, \,\, \theta^\prime(v)=0.  
\end{array}
$
\caption{"Surfaces" constructed from data which never enters the big cell.} 
\label{figure4}
\end{figure}

\noindent
The images in Figure \ref{figure4} are degenerate along the entire curve $u=0$. They 
are completely
degenerate in the big cell sense, because in one $s=t$ along the whole line,
and in the other $\theta^\prime=0$ along the whole line. The map $\hat \Phi$ never takes values in the big cell, and the map $f$ is just a curve.\\

\begin{figure}[ht]
\centering
$
\begin{array}{cc}
\includegraphics[height=35mm]{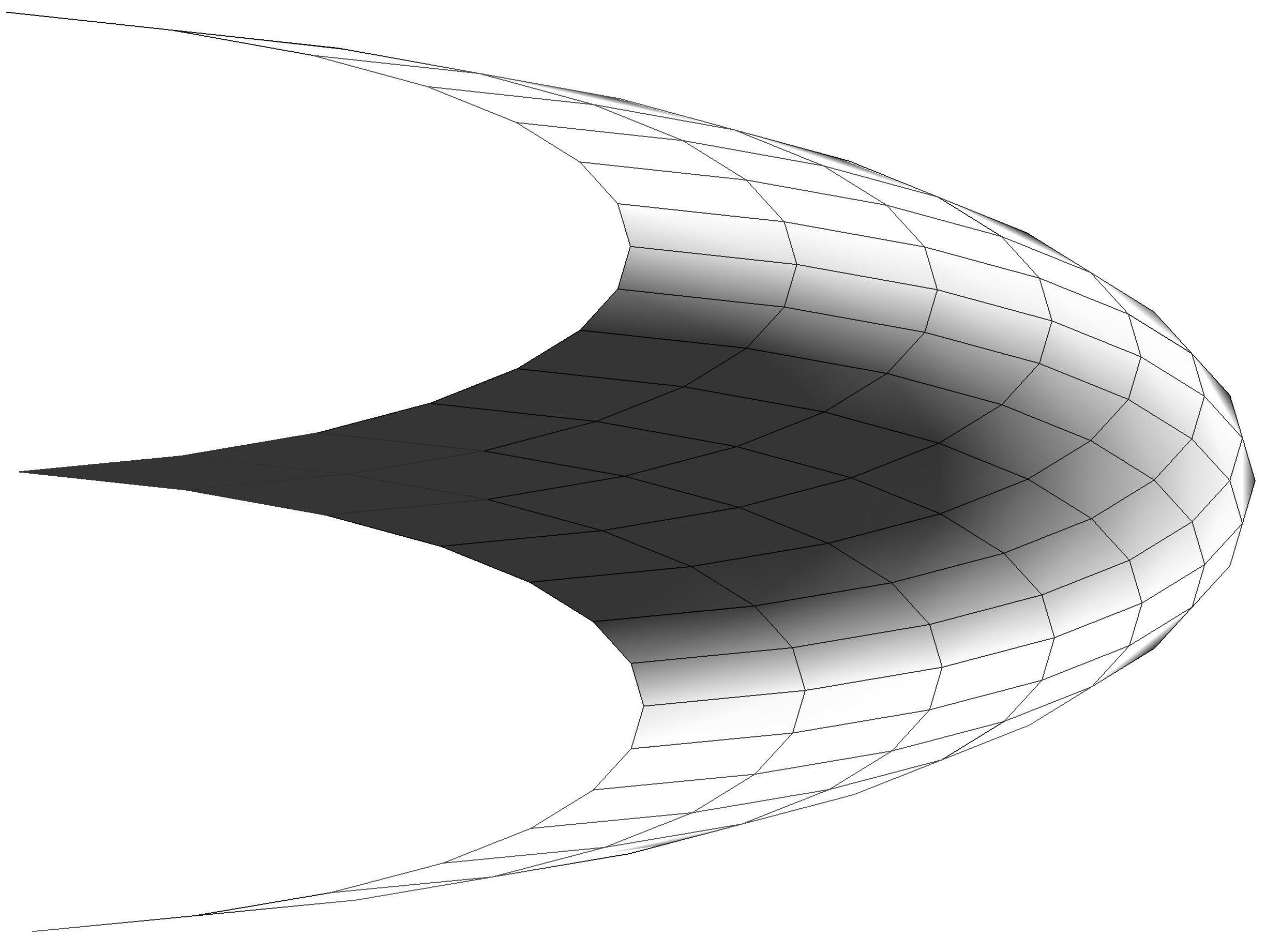} \quad & \quad
\includegraphics[height=35mm]{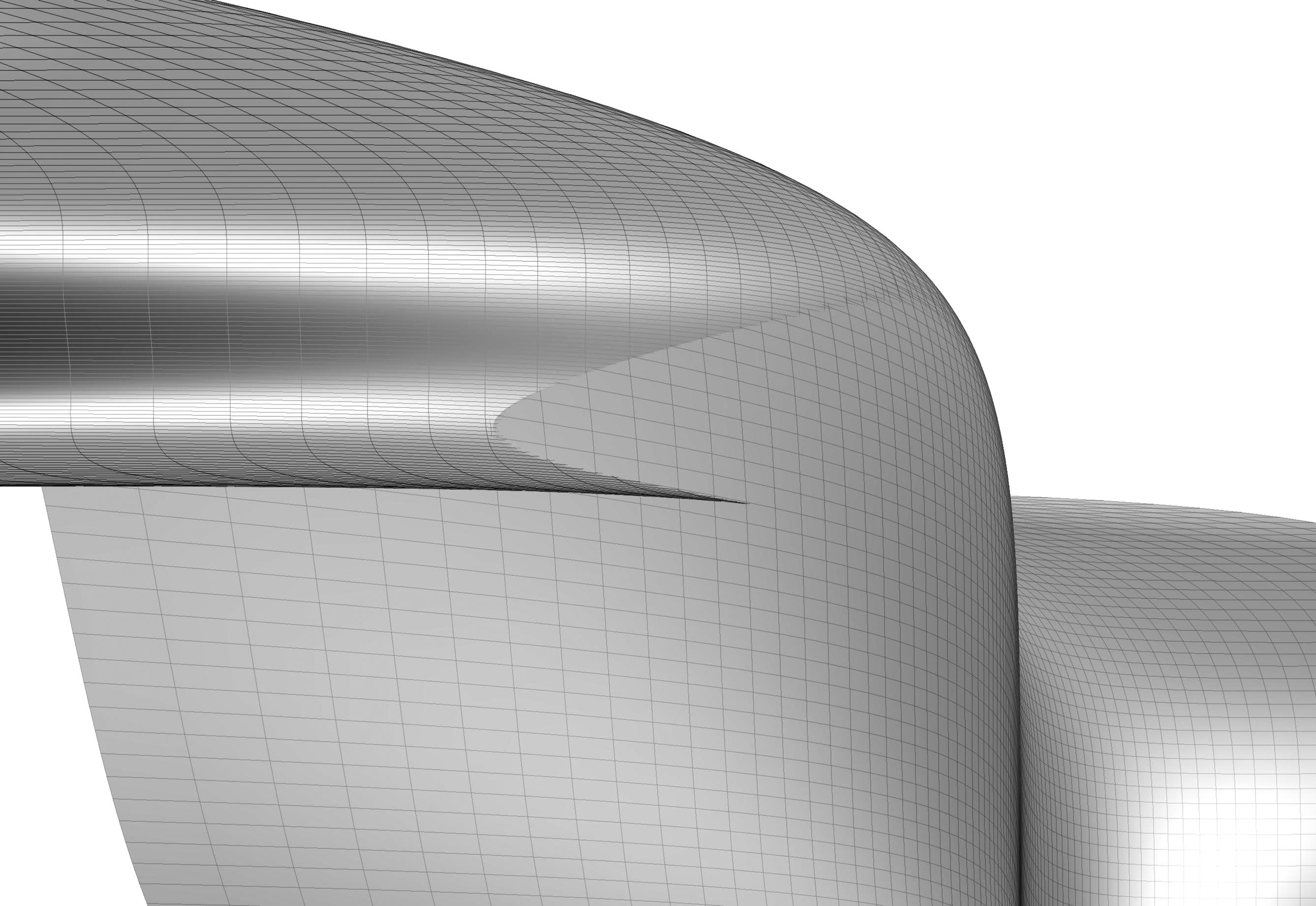} \\
 s(v)=1,\, \, t(v)=0,\,\, \theta^\prime(v)=0.0001 \quad & \quad
 s(v)=1,\, \, t(v)=2,\,\, \theta^\prime(v)=0.01v
\end{array}
$
\caption{Surfaces with degenerate singularities.} 
\label{figure5}
\end{figure}


\noindent
The first image in Figure \ref{figure5} is
 also degenerate along the whole line, because $t(v)=0$,
but this time only from the point of view of the theory of frontals.
The potential is non-degenerate, but not regular, which results in
a degenerate singularity (see Proposition \ref{nondegsingprop}). The surface
folds back over itself along the curve $u=0$, which is the curve along the
right hand side of this image.\\

\noindent
The last surface is degenerate only at the point $u=v=0$. It has the 
appearance of a cuspidal cross cap.

\providecommand{\bysame}{\leavevmode\hbox to3em{\hrulefill}\thinspace}
\providecommand{\MR}{\relax\ifhmode\unskip\space\fi MR }
\providecommand{\MRhref}[2]{%
  \href{http://www.ams.org/mathscinet-getitem?mr=#1}{#2}
}
\providecommand{\href}[2]{#2}

\end{document}